\documentclass[a4paper,parskip]{amsart}
\usepackage[utf8]{inputenc}
\usepackage[english]{babel} 
\usepackage{latexsym} 
\usepackage{graphicx}
\usepackage{stmaryrd}
\usepackage{bbold} 
\usepackage{amsmath,amscd,amssymb,amsfonts}
\usepackage{color}
\usepackage{pgfplots}
\usepackage[normalem]{ulem}
\usepackage{fancybox}
\usepackage{tikz}
\usepackage{empheq}
\usepackage{subcaption}
\usepackage{multirow}
\usepackage{verbatim}
\usepackage{caption}
\usepackage{algpseudocode}
\usepackage{lipsum}
\usepackage{algorithmicx,algorithm}
\usepackage{pst-node}
\usepackage{tikz-cd} 
\usepackage{microtype}
\usetikzlibrary{intersections,automata,arrows,positioning,calc,bending}
\usetikzlibrary{decorations.text}
\usepackage{ulem}
\usepackage{graphics,color}
\usepackage{tabularx}
\usepackage{kbordermatrix}
\usepackage{epsfig}
\usepackage{enumitem}

\def\Z{\mathbb{Z}} % integers
\def\N{\mathbb{N}} % natural numbers without 0
\def\calA{\mathcal{A}} %the alphabet
\def\calF{\mathcal{F}} %forbidden words
\def\calG{\mathcal{G}}
\def\sgn{\mathrm{sgn}}
\def\EE{\mathbf{E}} %replaces \mathfrak{E}}
\def\DD{\mathbf{D}} %replaces \DD}

\def\Isp{I_{\mathrm{sp}}} %separating "interval"
\def\tx{{\tilde x}}
\def\pmid{p_{\textnormal{mid}}}
\def\P{w}
\def\CardA{\mathfrak{a}}
\def\SL{S_{\textup{\textrm{L}}}} %left-moving bi-infinite conf.
\def\SLpos{S_{\textup{\textrm{L}}}^+} %left-moving right-infinite conf.
\def\SR{S_{\textup{\textrm{R}}}} %right-moving bi-infinite conf. 
\def\SRneg{S_{\textup{\textrm{R}}}^-} %right-moving left-infinite conf.
\def\SLppos{S_{\textup{\textrm{L}},p}^+} %left-moving right-infinite conf. which start at integer position p
\def\SRpneg{S_{\textup{\textrm{R}},p}^-} %right-moving left-infinite conf. which start at integer position p
\def\SLR{S_{\textup{\textrm{L}}}^{\textup{\textrm{R}}}} %conf. composed of right- and left-moving semi-infinte parts

\def\sigR{\sigma_{\textup{\textrm{R}}}} 
\def\sigRneg{\sigma_{\textup{\textrm{R}},-}}
\def\sigL{\sigma_{\textup{\textrm{L}}}}
\def\sigLpos{\sigma_{\textup{\textrm{L}},+}}
\newcolumntype{C}[1]{>{\centering\arraybackslash}m{#1}}
\def\Rrm{\textup{\textrm{R}}}
\def\Lrm{\textup{\textrm{L}}}
\def\OmConst{\Omega_{\mathrm{const}}} %was Omega_1
\def\OmVar{\Omega_{\mathrm{var}}} %was Omega_2
\def\f{f}
\def\fer{g_{e,r}}
\def\etaer{\eta_{e,r}}

\def\calF{\mathcal{F}}
\def\calZ{\mathcal{Z}}
\renewcommand{\emptyset}{\varnothing}

\usepackage{mdwlist}

\pgfplotsset{compat=1.12}

\newtheorem{theo}{Theorem}[section]
\newtheorem{prop}{Proposition}[section]
\newtheorem{lemma}{Lemma}[section]
\newtheorem{cor}{Corollary}[section]

\newtheorem{definition}{Definition}[section]
\newtheorem{remark}{Remark}[section]
\newtheorem{example}{Example}[section]

\title[1D Greenberg-Hastings cellular automata]{Dynamics and topological entropy of 1D Greenberg-Hastings cellular automata}
\author{M. Kesseb\"ohmer, J.D.M. Rademacher, D. Ulbrich}
\date{\today}
\thanks{\textsc{Mathematik, Universität Bremen, Bibliothekstr. 5, 28359 Bremen, Germany}}
\thanks{2010 \textit{Mathematics Subjects Classification} (Primary): 37B15; 37B20; 37B40}
\thanks{2010 \textit{Mathematics Subjects Classification} (Secondary): 74J35}
\thanks{This research was supported by the DFG grant RA 2788\textbackslash 1-1}

\begin{document}

\maketitle

%\recd{xx Month 20xx}

\begin{abstract}
In this paper we analyse the non-wandering set of 1D-Greenberg-Hastings cellular automata models for excitable media with $e\geqslant 1$ excited and $r\geqslant 1$ refractory states and determine its (strictly positive) topological entropy. We show that it results from a Devaney-chaotic closed invariant subset of the non-wandering set that consists of colliding and annihilating travelling waves, which is conjugate to a skew-product dynamical system of coupled shift-dynamics. Moreover, we determine the remaining part of the non-wandering set explicitly as a Markov system with strictly less topological entropy that also scales differently for large $e,r$. 
\end{abstract}
%\tableofcontents

\section{Introduction}
Following Greenberg, Hastings and Hassard \cite{GHH78}, we consider a basic cellular automaton model of an excitable medium based on the alphabet 
\[
\mathcal{A}\coloneqq  \left\{0,1,2,\ldots,e,e+1,e+2,\ldots,e+r\right\},
\] 
of cardinality $\#\mathcal{A}\coloneqq \text{card}(\mathcal{A}) =e+r+1 \eqqcolon\CardA$ for some positive intergers $e, r$.  Here, $E\coloneqq  \left\{1,2,\ldots,e\right\}$ represents the {\em excited states}, $R\coloneqq \left\{e+1,\ldots,e+r\right\}$ the set of {\em refractory states} and $0$ is the {\em equilibrium rest state}. The special case $e=r=1$ is  the most studied case (see in particular \cite{DS91}) even though the literature on this model is surprisingly scarce. However, the understanding of excitable media is of major importance in many different scientific contexts such as theoretical cardiology, neuroscience, chemistry, transition to turbulence, surface catalysis and it is a paradigm of nonlinear dynamics, self-organisation and pattern formation \cite{krinsky1991wave,Izhikevich}. Our main motivation stems from the problem of modeling  strong interaction of localised nonlinear waves in spatially extended partial differential equations. On the one hand, only in simple cases strong interaction can analytically be treated rigorously. Even the  major continuous models of excitable media, namely the FitzHugh-Nagumo-type systems, are not completely understood. On the other hand, in numerical simulation intricate spatio-temporal dynamics has been observed \cite{:/content/aip/journal/chaos/16/3/10.1063/1.2266993,pearson1993complex,petrov1994excitability,reynolds1994dynamics}, but there seems to be no rigorous analysis of its complexity.  

Unlike the  special case $e=r=1$, which has  been treated in  \cite{DS91}, we will show in this paper that in the case $e+r\geqslant 2$ the recurrent dynamics turns out to be in general much richer. We will provide a complete description of the recurrent dynamics: in addition to the pure pulse-annihilation dynamics, which is also present in the special case  $e=r=1$, there exists  an intricate Markovian structure caused by stationary dislocations and defects. We complete our observation by showing that in terms of  topological entropy the pulse-annihilation  dynamics has a strictly higher complexity than the Markovian structure. It turns out that both parts of the dynamics scale differently with respect to an increase of the cardinality of  excited and refractory states. We note that the topological entropy of CAs can have surprising properties and has been studied for several cases, e.g., \cite{hurd_kari_culik_1992,BLANCHARD199786, meyerovitch_2008}, though to our knowledge no general results can be applied here. For $e=r=1$ the entropy has been computed in  \cite{DS91} which serves as a guideline for the general case.

The cellular automaton model we study is a paradigmatic  model of excitable media that captures many of the  basic features. Let
\begin{equation*}
X\coloneqq \mathcal{A}^{\mathbb{Z}}=\left\{x=(\ldots,x_{-1},x_0,x_1,\ldots): x_k\in\mathcal{A}~\mbox{ for all } k\in\mathbb{Z}\right\}
\end{equation*}
denote the full $\mathcal{A}$-shift \cite{lind1995introduction}. Then $T\colon X\to X$ denotes the {\em cellular automaton}  given by
\begin{equation*}
T(\eta(x))=\EE(\eta(x))+\DD(\eta(x); \eta(x+1),\eta(x-1)),
\end{equation*}
where 
\begin{equation*}
\EE(k)=k+1,~1\leqslant k\leqslant e+r-1~~~\text{ and }~~~ \EE(e+r)=\EE(0)=0
\end{equation*}
and 
\begin{equation*}
\DD(u;v_1,v_2)=\begin{cases}1, & \text{ if }u=0\text{ and }1\leqslant v_i\leqslant e\text{ for }i=1\text{ or }2\\0, & \text{ otherwise}\end{cases}.
\end{equation*}
The function $\EE$ can be thought of as the \textit{reaction term} and $\DD$ models the {\em interaction} between neighbouring cells. If a cell is not at rest, then it evolves according to the reaction term. If, on the contrary, a cell is at rest then it becomes excited at level $1$ if and only if at least one of its two neighbors is excited. We equip $\mathcal{A}$ with the discrete topology and $X$ with the associated product topology which renders $X$ compact and $T$ continuous. 

As the system is translation invariant both in space and time it allows for relative equilibria, in particular travelling waves. The main building block of these are {\em `pulses'} of the form $x^*=(\ldots,0,0,1,2,\ldots, e+r,0,0,\ldots)$ and, spatially reflected, of the form $x^{**}$ with  $(x^{**})_{k}\coloneqq  (x^{*})_{-k}$, $k\in \Z$. These pulses travel to the left, or respectively to the right. More specifically,  if we define the left-shift   $\sigL\colon X\to X$ by $ (\sigL(x))_k=x_{k+1}$ and, analogously, the right-shift  $\sigR$, then we have
\[
T(x^*)=\sigL(x^*) \quad \mbox{ and } \quad T(x^{**})=\sigR(x^{**}).
\]
 The above  observation remains valid also for left-moving, respectively right-moving,   {\em multi-pulses}, that is for elements $x\in X$ given by  arbitrary concatenations of the finite word  $\P^\Lrm\coloneqq (1,2\ldots, e+r)$  (also referred to as  {\em local pulse})  and zeros, or $\P^\Rrm\coloneqq (e+r,e+r-1,\ldots,1)$ and zeros, respectively.  Due to the specific form of the coupling $\DD$ there is no dispersion, that is, the  distances between local pulses remain fixed within each multi-pulse. Hence the restriction of $T$ to the set of multi-pulses is conjugated to a left, respectively right, shift dynamical system. This observation gives further rise to the subsystem of counter-propagating semi-infinite multi-pulses. This invariant subsystem is determined by the key feature that pulses annihilate upon collision: in the simplest case consider an initial datum $x\coloneqq  (\ldots,0,0,\P^\Rrm,0_{2\ell},\P^\Lrm,0,0,\ldots)$, where $0_{2\ell}$ denotes a block of zeros of length $2\ell$. Then $T$ acts on $x$ by decrementing $\ell$ so that $T^{\ell}(x)=(\ldots,0,0, \P^\Rrm,\P^\Lrm,0,0,\ldots)$ and $T^{\ell+e+r}(x)=(\ldots,0,\ldots)$. We prove that the dynamics of counter-propagating semi-infinite multi-pulses with annihilation events constitutes a closed $T$-invariant Devaney-chaotic subset $Z\subset X $ referred to as the {\em pulse-annihilation dynamics}, cf. Fig. \ref{fig:subsystemZ}. Similar to \cite{DS91}, we combinatorially determine the topological entropy of $T$ restricted to $Z$ to be twice the entropy of the (sub)shift-dynamics on infinite pulse-trains. We show that on the corresponding subset $Z_\infty$, on which pulse-annihiliation never ends, the dynamics of $T$ is topologically conjugated to a skew-product dynamical system consisting of coupled shift-dynamics, thus giving us a heuristic understanding of the concrete value of the topological entropy.

 Since $T$ is a continuous endomorphism of a  compact metric space $X$, we know that the {\em non-wandering set} $\Omega$ of $T$ (cf. Def. \ref{nonwandering})   carries  the topological entropy of $T$, i.e. $h(X,T)=h(\Omega,T|_{\Omega}$). We shall show that the complexity of the dynamical systems is already determined by  further restricting to the pulse-annihilation dynamics  $(Z,T|_{Z})$.

The third author has shown in \cite{DU16} that  in the special case $e=r=1$ the non-wandering set $\Omega$, the pure pulse-annihilation system $Z$ and the  {\em eventual image}  $Y\coloneqq \bigcap_{n\in\N}T^n(X)$
 all coincide. In contrast, for $e+r>2$ both $\Omega$  and $Z$ are strict subsets of $Y$. In order to determine the topological entropy of $(X,T)$ we also have to study  the complement of $Z$ in the non-wandering set  $\Omega$, which, loosely speaking in terms of nonlinear wave phenomena, consists of stationary dislocations and defects, see Figure~\ref{fig:dislocations}. This  also leads to a complete  understanding of the structure of the recurrent dynamics.  
For this we introduce  transition graphs which determine the dynamics on $\Omega\setminus Z$, cf. Figure~\ref{fig:cc}. An explicit formula allows us to compare the entropy of $T$ restricted to $Z$ with the entropy of $T$ restricted to its complement in $\Omega$.  These explicit formulae allow us to  study the limits as $e,r\to\infty$. On the one hand, asymptotically and after time-rescaling, the limiting entropy of the whole system is twice the topological entropy of the \emph{full shift} over an alphabet with $e+r$ symbols. On the other hand, the limit of the restriction to $\Omega\setminus Z$ depends strongly on the difference of $r$ and $e$. 

This paper is organised as follows.  In Section \ref{sec:preliminaries}  we provide the basic setting and introduce the necessary  notations. Next, we focus on the pure pulse-annihilation  subsystem  in  Section \ref{sec:collisions} and its skew-product representation. In the main part, Section \ref{s:entropy}, we give a detailed analysis of  the non-wandering set and determine its topological entropy, including its asymptotics. We end with a short outlook and discussion.    
%%%%%%%%%%%%%%
\section{Preliminaries and notation}\label{sec:preliminaries}
In this section, we provide the topological setup of our model and introduce some notational conventions as used in symbolic dynamics and throughout this paper.            

\subsection{Topological setting}
We recall that the topology on $X$ is generated by the clopen \textit{cylinder sets}, for $a_{0},\ldots,a_{n}\in \mathcal{A}$, $n\in \N$ and $m\in \Z$, 
\[
[a_0,\ldots,a_n]_m\coloneqq \left\{x\in X: x_m=a_0,\ldots,x_{m+n}=a_n\right\}.
\] 
The topological space $X$ is compact and metrisable, and a metric inducing the topology is, for $x,y\in X$  e.g.\  given by
\begin{equation}\label{eq1}
d(x,y)=\begin{cases}2^{-k} & \text{if }x\neq y\text{ and }k\text{ is maximal so that }x_{[-k,k]}=y_{[-k,k]}\\0 & \text{if }x=y\end{cases}
\end{equation}
with the convention that if $x=y$ then $k=\infty$ and $2^{-\infty}=0$, while if $x_0\neq y_0$ then $k=-1$. See  \cite{lind1995introduction} for further details. Moreover, $T\colon X\to X$ is continuous with respect to the product topology.

The concept of \textit{topological entropy} was first introduced for continuous self-maps of compact metric spaces by Adler, Konheim and McAndrews \cite{Adler} and is a widely accepted measure of the complexity of a dynamical system. Bowen and Dinaburg \cite{bowen1971entropy, dinaburg1970correlation} gave a definition for uniformly continuous maps on (not necessarily compact) metric spaces which coincides with the previous definition in case of compactness. It was shown that the latter definition works for any continuous self-map whenever the metric on the space is totally bounded \cite{hasselblatt2005topological}. In this paper, the topological entropy of $T$ on $X$ will be denoted by $h(X,T)$. Both the eventual image $Y$ and the non-wandering set $\Omega$ determine the topological entropy, i.e. $h(X,T)=h(Y,T|_{Y})=h(\Omega, T|_{\Omega})$ \cite{walters2000introduction}. 

The eventual image is characterised as a shift space $X_\calF$ with respect to forbidden blocks $\calF$ \cite{KRU}. In this context, it is crucial to notice that $Y$ is the set of all configurations $x\in X$ with $T^{-n}(x)\neq\emptyset$ for all $n\in\N$ and that $T\colon X\to X$ is not surjective: the preimage at a lattice site $j\in\Z$ is $T^{-1}_j:X\to \calA\cup\{\{0,e+1\}\} \cup \emptyset$ with
\begin{equation}\label{pre-image}
T^{-1}_j(x) := \begin{cases}
x_j-1,& \mbox{ if } x_j>1,\\
0, & \mbox{ if } x_j=1 \mbox{ and }x_{j-1}\in E+1  \mbox{ or }x_{j+1}\in E+1,\\
e+r, & \mbox{ if } x_j=0 \mbox{ and } x_{j-1}\in E+1  \mbox{ or }x_{j+1}\in E+1,\\
\{0,e+r\}, & \mbox{ if } x_j=0 \mbox{ and } x_{j\pm1}\not\in E+1,\\
\emptyset, & \mbox{ if } x_j=1 \mbox{ and } x_{j\pm1}\not\in E+1,
\end{cases}
\end{equation}
where $E+1=\{a+1: a\in E\}$; note $\calA\setminus (E+1) = \{0,1,e+2,\ldots, e+r\} = (R+1)\cup \{1\}$ with addition  mod $\mathfrak{a}$.
\subsection{Symbolic dynamics}
We refer to  $(a_1,a_2,\ldots,a_k)\in\calA^{k}$ as a {\em block} (or {\em word}) over $\mathcal{A}$. The elements in $X$ are also referred to as  bi-infinite blocks.   The length $\lvert w\rvert$ of a block $w$ is the number of symbols it contains, i.e. a $k$-block $w$ is a block of length $\lvert w\rvert=k$. In particular, we use the notation $0_k\coloneqq (0,\ldots, 0)\in \mathcal{A}^{k}$ for the $k$-block consisting only of zeros and $0_\infty, 0_\infty^\pm$ for the (semi-) infinite zero blocks. Note that the action of $T$ naturally carries over to the set of blocks.

For $i,j\in\Z$ with $i\leqslant j$, we denote the block of coordinates in $x$ from position $i$ to position $j$ by $x_{[i,j]}=(x_i,x_{i+1},\ldots,x_j)$. If $i>j$, $x_{[i,j]}$ is the empty block, denoted by $\emptyset$. For convenience, if $i,j\in\{\pm\infty\}$ and $p\in\Z$, we stick to this notation by setting $x_{[-\infty,p]}\coloneqq  (\ldots,x_{p-1},x_{p})\in\calA^{\Z_{\leqslant p}}, x_{[p,\infty]}\coloneqq (x_{p},x_{p+1},\ldots )\in\calA^{\Z_{\geqslant p}}$ and $x_{[-\infty,\infty]}\coloneqq  x \in X$.

If $w$ is a block, we say that $w$ occurs in $x\in X$ (or that $x\in X$ contains $w$) if there are indices $i$ and $j$ such that $w=x_{[i,j]}$. A subblock of a block $w=(a_1,a_2,\ldots,a_k)$ is a block of the form $v=(a_i,a_{i+1},\ldots,a_j)$ where $1\leqslant i\leqslant j\leqslant k$ and we also say that $v$ occurs in $w$ or that $w$ contains $v$ and write $w=(a_1,a_2,\ldots,a_{i-1},v,a_{j+1},a_{j+2},\ldots,a_k)$.

For finite blocks $w=(w_1,\ldots,w_m), v=(v_1,\ldots,v_n)$ and a configuration $x\in [w,v]_{p-m+1}:=[w_1,\ldots,w_m,v_1,\ldots,v_n]_{p-m+1}$, we use the notation 
\begin{equation*}
(w~^p|v)=x_{[p-m+1,p+n]}
\end{equation*}
in order to specify the position $p\in\Z$ at which the two blocks are linked. Slightly abusing this notation, if $w\in\bigcup_{q\in\Z}\calA^{\Z\leqslant q}$ is a left-infinite configuration and $v$ is a finite block of length $n\in\N$, we write $x=(w~^p|v)$ for the the left-infinite configuration $x\in\calA^{\Z\leqslant p+n}$ with $x_{[-\infty,p]}=w$ and $x_{[p+1,p+n]}=v$ (and analogously for $w$ and $v$ being a finite block and a right-infinite configuration, respectively). If both $w$ and $v$ are semi-infinite, $x=(w~^p|v)$ denotes the configuration $x\in X$ with $x_{[-\infty,p]}=w$ and $x_{[p+1,\infty]}=v$.
In the same sense, we simply write
\begin{equation*}
(w,v)
\end{equation*}
if the position is irrelevant.

To express the temporal dynamics of $T$, we use the transposed block notation
\begin{equation*}
w^\intercal=(w_1,\ldots,w_n)^\intercal\coloneqq \begin{pmatrix}w_n\\\vdots\\w_1\end{pmatrix}.
\end{equation*}
A configuration $x\in X$ or a block $\omega=x_{[i,j]}$ occuring in $x$ is $n$-periodic if $n\geqslant 1$ is the smallest integer such that $T^n(x)=x$, or $(T^n(x))_{[i,j]}=\omega$, respectively.

In the sequel, we need some notion of `distance' between states $a,b\in\mathcal{A}$. To this end, we make the convention to consider $\mathcal{A}$ as the group $(\mathbb{Z}/\CardA\mathbb{Z},+)$ while inequalities involving elements in $\calA$ are meant with respect to $\Z$.
\begin{definition}
For $a,b\in\mathcal{A}$, let $\calA\ni s(a,b)\coloneqq b-a$ be the \textit{step size from $a$ to $b$}. 
\end{definition}
For illustration, it is convenient to think of the corresponding alphabet
\begin{equation}\label{alph2}
\mathcal{A}'\coloneqq \{a'=e^{i\varphi(a)}: a\in\mathcal{A}\},\qquad\varphi(a)\coloneqq {2\pi a}/{\CardA}
\end{equation}
on $\{x\in\mathbb{C}: \lvert x\rvert=1\}$, cf. Figure~\ref{unitcircle}. 
\begin{figure}[htbp]
\centering
%\begin{center}
\begin{tikzpicture}[scale=6.0,cap=round,>=latex]
  \def\Radius{.3cm}
  \draw (0cm,0cm) circle[radius=\Radius];
  \begin{scope}[
    -{Stealth[round, length=8pt, width=8pt, bend]},
   shorten >=4pt,
    very thin,
  ]
    \draw (0,-\Radius) arc(270-3:270+3:\Radius);
    \draw (0,\Radius) arc(90-3:90+3:\Radius);
  \end{scope}
  % draw the two points 
  \fill[radius=0.3pt]
    (0:\Radius) circle[] node[below right](0) {$0'$}
    (-45:\Radius) circle[] node[below right] {$(e+r)'$}
   (-180:\Radius) circle[] node[above left] {$\left(\frac{\CardA}{2}\right)'$}

   (-225:\Radius) circle[] node[above left](e){$e'$}
  ;
  \def\Item#1#2(#3:#4){%
    \path[
      decoration={
        text along path,
        text={#1'},
        text align=center,
      },
      decorate,
    ]
      (#3:\Radius-#2) arc(#3:#4:\Radius-#2)
    ;
  }
  \Item E 2pt (180:-40)
  \Item R 1pt (240:300)

  \draw[dashed] (0:\Radius) -- (180:\Radius);
  \draw[dashed] (0,0) -- (e);
  \draw (0.1,0) arc (0:135:0.1);
  
    \node[] at (0.02,0.05)  {$\varphi(e)$};
    \node[] at (\Radius, \Radius) {arcmin(0',e')};
\end{tikzpicture}
%\end{center}
\caption{Illustration of $\mathcal{A}'$ in case $e\leqslant r$. $\varphi(e)<\pi$ and $\varphi\left({\CardA}/{2}\right)=\pi$.}
\label{unitcircle}
\end{figure}
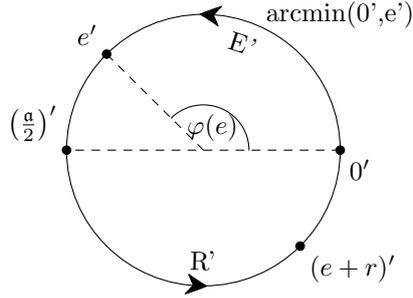

%%%%%%
\section{The pulse collision subsystem}\label{sec:collisions}
We first consider the aforementioned invariant subsystem for which we can compute the topological entropy explicitly and infer other features of the dynamics. This will already give a lower estimate for the topological entropy and complexity of the whole system. This subsystem depends on $e+r$ only and thus is a lower complexity bound independent of the order of $e$ and $r$.

The set of multi-pulses and infinite wavetrains mentioned in the introduction forms a subshift constructed by the transition graph plotted in Figure~\ref{f:trans} (left): adjacent block entries differ by one or are both zero, i.e., $s(x_i,x_{i+1})\in \{\sgn(x_i),1\}$ or $s(x_{i+1},x_{i})\in \{\sgn(x_{i+1}),1\}$. The transition matrix $A\in\{(a_{i,j})_{i,j=1}^{\CardA}: a_{i,j}\in\{0,1\}\}$ corresponding to this graph is shown in Figure~\ref{f:trans} (right).
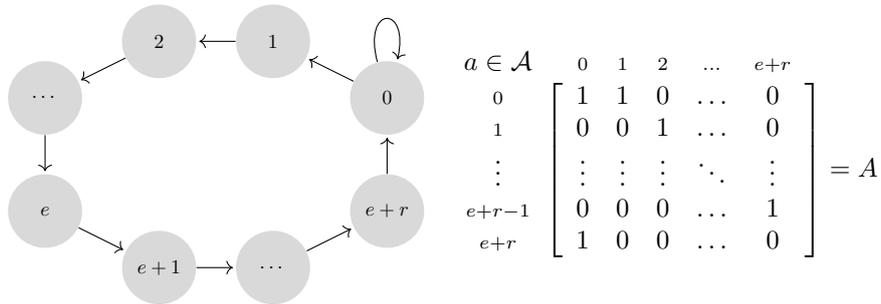
\begin{figure}[H]
%\begin{center}
\centering
\begin{tikzpicture}[baseline={([yshift=-2ex]current bounding box.center)},scale=0.75, transform shape,shorten >=1pt,->]
  \tikzstyle{vertex}=[circle,fill=black!15,minimum size=37pt,inner sep=0pt]
  \node[vertex] (3) at (0,0) {$\cdots$};
  \node[vertex] (2) at (2,1)    {$2$};
  \node[vertex] (1) at (4,1) {$1$};
  \node[vertex] (0) at (6,0) {$0$};
  \node[vertex] (e+r) at (6,-2) {$e+r$};
  \node[vertex] (6) at (4,-3) {$\cdots$};
  \node[vertex] (e+1) at (2,-3) {$e+1$};
  \node[vertex] (e) at (0,-2) {$e$};
  \path (0) edge[->] (1);
  \path (1) edge[->] (2);
  \path (2) edge[->] (3);
  \path (3) edge[->] (e);
  \path (e) edge[->] (e+1);
  \path (e+1) edge[->] (6);
  \path (6) edge[->] (e+r);
  \path (e+r) edge[->] (0);
  \path (0) edge[->, loop above] (0);
\end{tikzpicture} \;\;
$\displaystyle{
\kbordermatrix{\mbox{$a\in\mathcal{A}$}&0&1&2&\hdots&e+r\\
0&1 & 1 & 0 &\hdots&0\\
1&0&0&1&\hdots&0\\
\vdots&\vdots&\vdots&\vdots&\ddots&\vdots\\
e+r-1&0 &0&0&\hdots&1\\
e+r&1&0&0&\hdots&0
}=A}$
%\end{center}
\caption{Transition graph for the one-sided subshift defining the pulse system.}
\label{f:trans}
\end{figure}

Accordingly, we define sets of right- and leftmoving infinite multi-pulse-type solutions, 
\begin{equation*}
\SR\coloneqq \{x\in \mathcal{A}^{\mathbb{Z}}: a_{x_i,x_{i-1}}=1,\;i\in\Z\},\quad
\SL\coloneqq \{x\in \mathcal{A}^{\mathbb{Z}}: a_{x_i,x_{i+1}}=1,\; i\in \Z\},
\end{equation*}
their semi-infinite analoga,
\begin{align*}
\SRneg\coloneqq &\bigcup_{p\in\Z}\SRpneg,\quad \SRpneg\coloneqq \left\{x\in\mathcal{A}^{\Z_{\leq p}}:a_{x_i,x_{i-1}}=1\right\},\\
\SLpos\coloneqq &\bigcup_{p\in\Z}\SLppos,\quad \SLppos\coloneqq \left\{x\in\mathcal{A}^{\Z_{\geq p}}: a_{x_i,x_{i+1}}=1\right\},
\end{align*}
and infinite configurations composed of semi-infinite counter propagating parts,
\begin{equation*}
\SLR\coloneqq \bigcup_{p\in\Z}\{(x^\textrm{R}~^p| x^\textrm{L}) \in \SRpneg\times \textrm{S}_{\textrm{L},p+1}^+\} \setminus(\SR\cup\SL).
\end{equation*}
In the definition of $\SLR$ we have identified pairs in $\SRpneg\times S_{\textrm{L},p+1}^+$ with configurations in $X$ by gluing these together at position $p$. Conversely, for given $x\in \SLR$ this position is not unique, but there are only finitely many options $p\in\Z$ such that $(x_{[-\infty,p]}, x_{[p+1,\infty]})\in\SLR$ since $x\notin \SR\cup\SL$. This motivates the following notion.

\begin{definition}
For \textup{$x\in \SLR\cup\SR\cup\SL$} we call $p\in \overline{\Z}\coloneqq \Z\cup\{\pm\infty\}$ a separating position if \textup{$(x_{[-\infty,p]}, x_{[p+1,\infty]})\in \SLR$} or either \textup{$x_{[-\infty,p]}\in\SR$} with $p=\infty$ or \textup{$x_{[p,\infty]}\in\SL$} with $p=-\infty$. We denote the set of separating positions of $x$ by $\Isp(x)$ or simply $\Isp$ if $x$ is clear from the context. 
\end{definition}

Note that $\#\Isp=1$ for $x\in \SR\cup\SL$ and $\#\Isp<\infty$ for $x\in\SLR$. $T$ acts as the right-shift $\sigR$, $(\sigR(x))_{j+1}=x_j$, on $\SR$ and on $\SL$ as the left-shift $\sigL$, so that $\SR, \SL$ form invariant subsets of $X$, cf. Figure~\ref{fig:rightshift}. Conversely, $T$ acts as a shift on such configurations only. On finite blocks of type $\SR$ or $\SL$ the map $T$ acts on $x_{[j_-+1,j_+-1]}$ as the corresponding shift, and thus locally in space. On $\SRneg$ and $\SLpos$ the shifts are denoted by $\sigRneg, \sigLpos$, respectively.
\begin{figure}[t]  \centering
\begin{tikzpicture}
\node[above right] (img) at (0,0) {\includegraphics[scale=0.4]{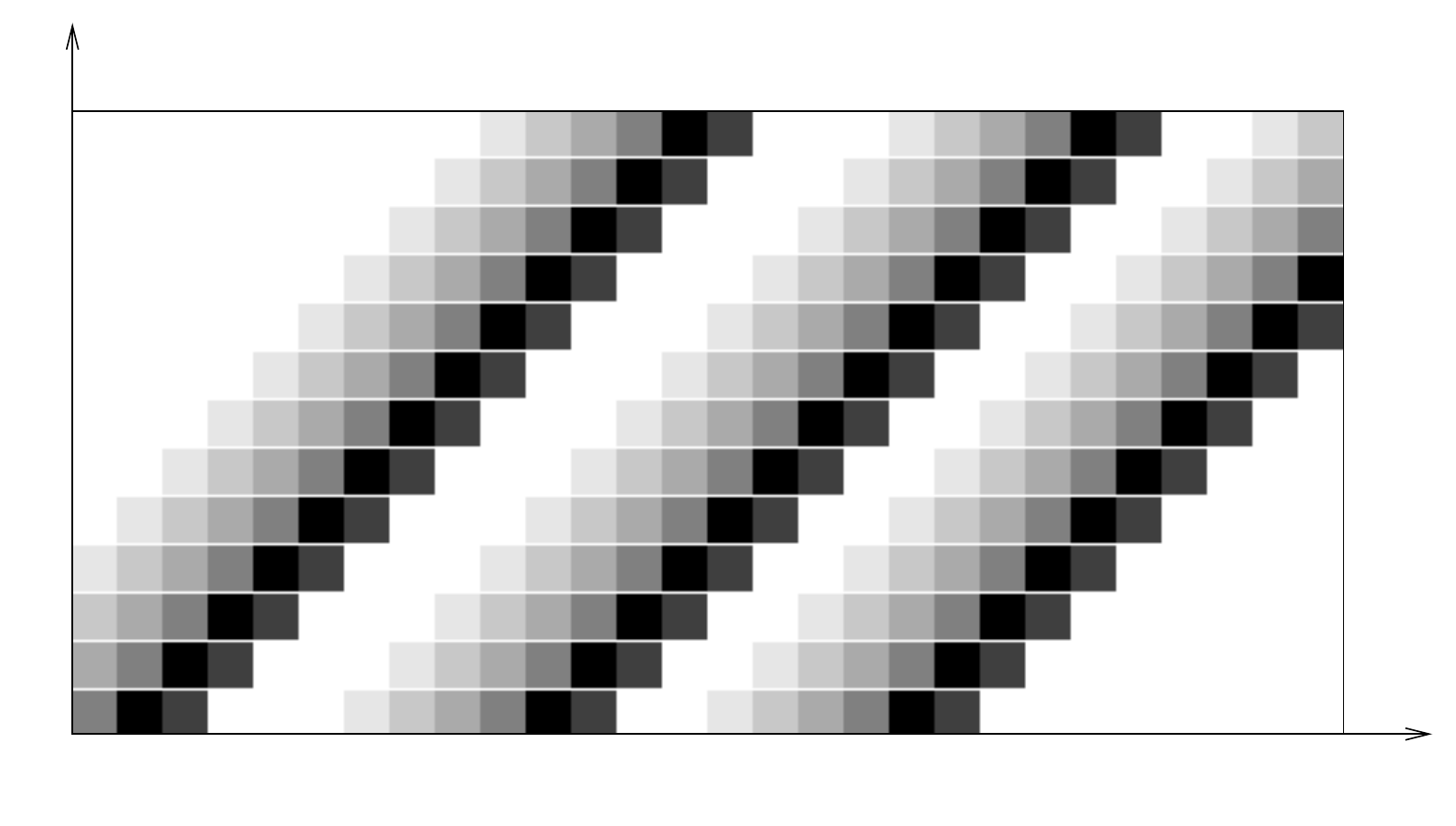}};
\node at (170pt,10pt) {space};
\node[rotate=90] at (6pt,90pt) {time};
\end{tikzpicture}
\caption{Snapshot of $T$ acting on $\SR$ as right shift $\sigR$}
\label{fig:rightshift}
\end{figure}

We now define the key notion of this section: 
\begin{definition}
The pulse collision subsystem \textup{$Z\subset \SL\cup\SR\cup\SLR$} is the set of $x\in X$ with either \textup{$x\in \SR\cup\SL$}, or \textup{$x\in\SLR$} such that for $p=\max \Isp(x)$ we have  $x_{p+1}=x_p$ or \textup{$(x_{[-\infty,p]},x_{[p,\infty]})\in\SLR$}. 
Let $Z_\infty\subset Z\cap\SLR$ be the subset of configurations $x\in \SLR$ for which $x_{[-\infty,k]}\neq 0_\infty^-, x_{[k,\infty]}\neq 0_\infty^+$ for all $k\in\Z$.
\end{definition}

As will be discussed more in the following, $Z$ consists of configurations which are either purely left- or rightmoving under the dynamics of $T$, or sequences of local pulses $\P^\textrm{L}$  (leftwards) and $\P^\textrm{R}$ (rightwards) glued at one position, which annihilate each other in time. That is, the dynamics on $Z$ completely consist of pulse dynamics and pulse annihilation, cf. Figure~\ref{fig:subsystemZ}. The set $Z_\infty\subset\SLR$ captures the configurations for which the collision and annihilation never ends.
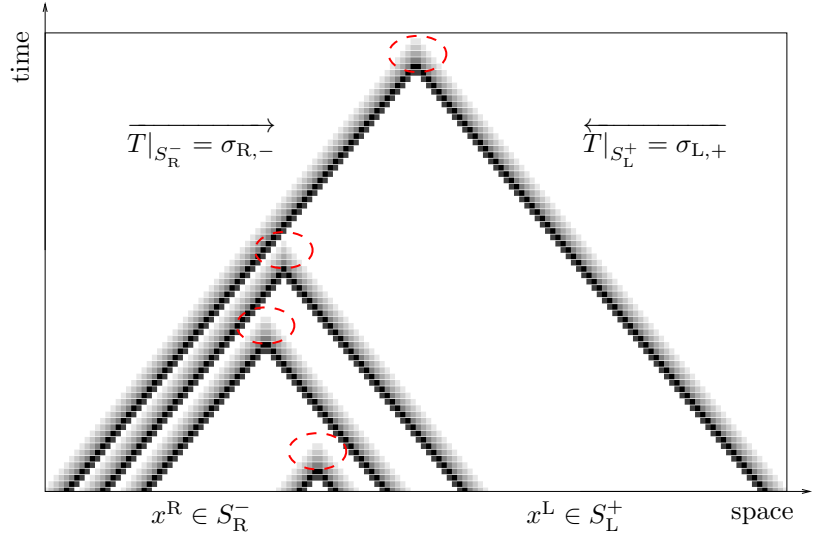
\begin{figure}[!h]   \centering
  \begin{tikzpicture}
\node[above right] (img) at (0,0) {\includegraphics[scale=0.8]{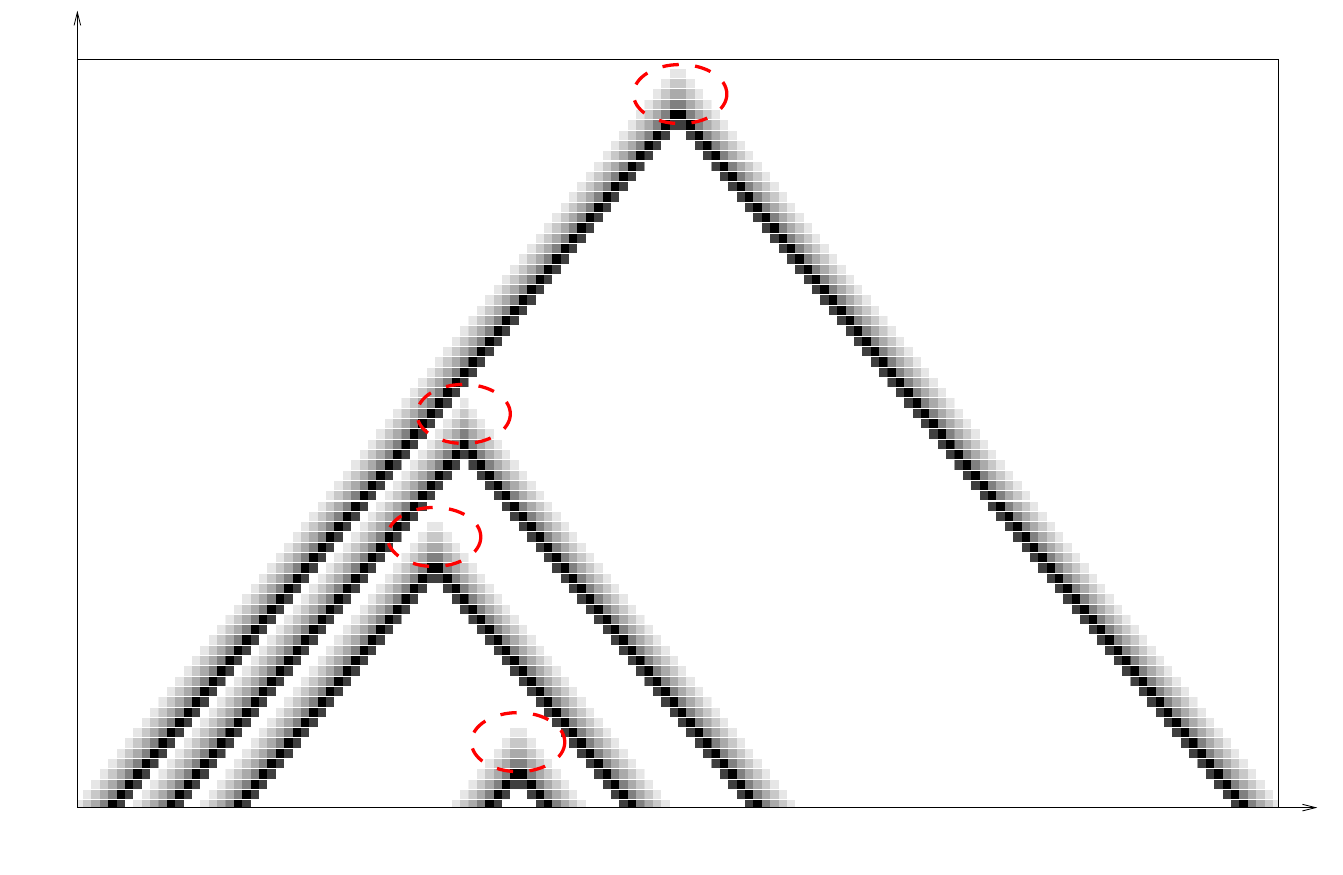}};
\node at (250pt,150pt) {$\overleftarrow{T|_{\SLpos}=\sigLpos}$};
\node at (80pt,150pt) {$\overrightarrow{T|_{\SRneg}=\sigRneg}$};
\node at (220pt,10pt) {$x^\Lrm\in\SLpos$};
\node at (80pt,10pt) {$x^\Rrm\in\SRneg{}$};
\node at (290pt,10pt) {space};
\node[rotate=90] at (12pt,183pt) {time};
\end{tikzpicture}
\caption{Space-time plot of GHCA simulation of $x=(x^\textrm{R},x^\textrm{L})\in\SLR$. Marked is a sequence of four pulse annihilation events for $e=2$ and $r=4$; the rest state is shown in white, the gray levels increase in order $e+r, e+r-1,\ldots,e+1,1,\ldots,e$.}
\label{fig:subsystemZ}
\end{figure}
More specifically, for $x\in Z$ with separating position $p\in \Z$ there is $(x^\textrm{R}, x^\textrm{L})\in\SLR$ such that $x=(x^\textrm{R}~^p|x^\textrm{L})$ and some $\tx\in T^{-n}(x)\in \SLR$ for some minimal $n\in\N_0$ such that $\tx$ is of pre-collision type, i.e. $\tx=(\tx^\textrm{R},0_\ell,\tx^\textrm{L})$ with $\tx^\textrm{R}=(\ldots,2,1)$, $\tx^\textrm{L}=(1,2,\ldots)$ and separating interval $\Isp=[p_-,p_+]$ with $\tx_{[p_-,p_+]}=(1,0_\ell)$. Recall the dynamics of iterating $T$ for even $\ell=2\ell'$ decrements $\ell'$ until $\ell=0$, i.e., 
\[
T^j(x) = (x^\textrm{R},0_{2(\ell'-j)},x^\textrm{L}), \quad x_{[-\infty,p_-]}=x^\textrm{R}, \quad j=0,1,\ldots,\ell'. 
\]
After this the annihilation takes place for $T^{\ell'+j}(x)$, $j=0,1,\ldots,e+r$, which is characterised by a unique and constant separating position $p$ and by $x_p=x_{p+1}$. These values are incremented until both are $e+r$ and the next time step leads to $x_p=x_{p+1}=0$ so the pulses have annihilated each other. 

For odd $\ell$ the annihilation starts via
\begin{align*}
T(\ldots,2,1,&0,1,2,\ldots)\\
=(\ldots,3,2,&1,2,3, \ldots),
\end{align*}
where dots denote some continuation in $\SR, \SL$, respectively.
The central block $(1,0,1)$ in the preimage determines the separating positions, which lies either at the leftmost $1$ or at the $0$, so in contrast to annihilation for even $\ell$ here we have two separating positions, $\Isp=[p,p+1]$, where $p+1$ lies at the center of the block $(1,0,1)$ in the preimage. Further iteration yield the annihilation analogous to the even case, except there are two separating positions.
Notably, at collision (i.e., when there are no zeros at separating positions) and during annihilation we have $x_{[p,\infty]}, x_{[p+1,\infty]}\in\SL$.

\begin{remark}\label{r:sepZ}
In summary, the pre-collision type configuration
\begin{equation*}\label{pre-collision}
x=(x^\Rrm,0_\ell,x^\Lrm)\textnormal{ with }(x^\Rrm,x^\Lrm)=(\ldots,2,1,1,2,\ldots)\in\SLR
\end{equation*}
has $\Isp=[p_-,p_+]$, and $x_{[p_-,p_+]}=(1,0_\ell)$. For even $\ell>0$  the number of separating positions is odd and the collision is characterised by a block $(2,1~^p|1,2)$ with unique separating position $p$, i.e.,$\#\Isp=1$. This remains the unique separating position for  $e+r$ time steps. 
For odd $\ell$ the number of separating positions is even and pulse collision occurs at a block $(2,1,2)$ which yields two separating positions, $\#\Isp=2$, also during annihilation.\\
In the cases $x^\textrm{R}\equiv 0$ or $x^\textrm{L}\equiv 0$ and more generally for $x\in\SR\cup \SL$ the separating positions also form an interval $\Isp=[p_-,p_+]$, with $p_+=\infty$ or $p_-=-\infty$, or $\Isp=[\pm\infty]$ for  one sign.
\end{remark}

\begin{lemma}\label{l:invariants}~
\begin{itemize}
\item[(i)] $T(\SR)=\SR$, $T(\SL)=\SL$, and  $T|_{\SR}={\sigR}|_{\SR}, T|_{\SL}={\sigL}|_{\SL}$ are invertible.
%$T\colon\SR\to\SR$ and $T\colon\SL\to\SL$ are invertible and $T|_{\SR}={\sigR}|_{\SR}, T|_{\SL}={\sigL}|_{\SL}$.
%
\item[(ii)] $T(Z)= Z\subset Y$, $T(Z_\infty)=Z_\infty$,
\item[(iii)] Let $x\in Z$. If $\# \Isp=1$ or $x_{p+1}\neq 0$ with $\Isp=[p,p+1]$, then $x$ has a unique preimage under $T$ in $Z$. Otherwise there are $\#\Isp$ possible preimages in $Z$.
\item[(iv)] $Z$ is closed (hence compact) and $Z_\infty$ is not closed with closure $\overline{Z_\infty}=Z$.
\end{itemize}
\end{lemma}
\begin{proof}
(i) As noted above $\SR$ and $\SL$ are naturally forward invariant under $T$. For $x\in\SR$ there is a unique preimage in $\SR$ being the left or right shift, respectively. More precisely, each $1$ lies in a block $(2,1,0)$ and each $0$ lies in the center of a block $(0,0,0)$ or $(1,0,0)$ with unique preimage in $\SR$ being $0$, or it lies in a block $(0,0,e+r)$ with unique preimage in $\SR$ being $e+r$. Analogously this holds for $T|_{\SL}:\SL\to\SL$.

(ii) The argument in (i) and the annihilation procedure discussed above shows that the image of $Z$ lies in $Z$, and also that each point in $Z$ has a preimage in $Z$, which means $Z\subset Y$.  Likewise for $Z_\infty$.

(iii) The case $\#\Isp=1$ implies either $x\in \SR\cup\SL$ in which case the preimage in $Z$ is unique as in item (i). Otherwise $x=(x^\textrm{R},x^\textrm{L})\in\SLR$, $x^\Lrm\in\SL$ as in Remark~\ref{r:sepZ} for $\ell=0$ with $x_p=x_{p+1}>0$ and the unique preimage in $Z$ is $(x^\Rrm,0,0,x^\Lrm)$ with zeros at positions $p$ and $p+1$. The case of two separating positions corresponds to the annihilation from odd $\ell$ and again $x=(x^\Rrm,x^\Lrm)$ with $x_p=x_{p+1}+1>1$ by assumption so that the unique preimage in $Z$ is $(x^\Rrm,x_p-1,x^\Lrm)$ with $x_p-1$ at position $p$.

Finally, if $\Isp=[p_-,p_+]$ with $p_+-p_->2$ then $x_{[p_-,p_+]}=(1,0_\ell)$ for $\ell=p_+-p_--1$, cf.\ Remark~\ref{r:sepZ} and $x=(x^\Rrm,0_\ell,x^\Lrm)$ with $x_{[-\infty,p_-]}=x^\Rrm$. The possible preimages are $x^0=(x^\Rrm,0_{\ell+2},x^\Lrm)$ and $x^j=(x^\Rrm,0_j,e+r,0_{\ell-j+1},x^\Lrm)$ with $x^0_{[-\infty,p_--1]}=x^\Rrm$ and $j=1,\ldots,\ell$ which makes $\#\Isp=\ell+1$.

(iv) By construction of elements in $Z$ it contains limits of converging sequences. Since such limits of a sequence in $Z_\infty$ may be, e.g., the zero sequence, $Z_\infty$ is not closed. However, any point in $Z$ can be approximated in the cylinder topology with a sequence in $Z_\infty$ by replacing the infinite tails with tails of elements from $Z_\infty$.
\end{proof}

%%%%%%%%%
\subsection{Topological entropy on $Z$ and its asymptotics}\label{sec:TEofZ}

We determine the topological entropy of $(Z,T|_{Z})$. The proof is an adaption and more complete exposition of the technique in \cite{DS91} for the case $e=r=1$, which is special as $Y=Z$, which has been shown by the third author in \cite{DU16}, i.e., eventual image and pulse collision subsystem coincide in this case only. We remark that it is purely a combinatorial counting argument of space-time window and in this sense independent of the topology.

\begin{prop}\label{p:entropy}
For arbitrary    $e,r\geqslant 1$ the topological entropy of $T$ restricted to $Z$  is given by $h(Z,T|_{Z})=2\ln\rho_{e+r}$, where  $\rho_{e+r}$ denotes the largest eigenvalue of $A$, which is the positive real root of $\lambda^{e+r+1}-\lambda^{e+r}-1$. In particular, $h(X,T)\geqslant 2\ln\rho_{e+r}$.
\end{prop}
\begin{proof}
We determine a substitution for $Z$ by (essentially) replacing each pulse travelling right (left) with an `$r$' (`$\ell$') symbol at the $e+r$ state and all other states by zeros; thus between two $r$'s and between two $\ell$'s there are at least $e+r$ zeros.  Specifically, let $Z'\subset\{0,r,\ell\}^{\mathbb{Z}}$ be the set of configurations $z$ for which there exists some $p\in\mathbb{Z}\cup\{\pm\infty\}$ such that 
\begin{enumerate}
\item $z_i\in\{0,r\}~\forall i\leqslant p$, and $\lvert i-j\rvert>e+r~\forall i,j\in\Z_{\leqslant p}, i\neq j: z_i=z_j=r$,
\item $z_i\in\{0,\ell\}~\forall i>p$, and $\lvert i-j\rvert >e+r~\forall i,j\in\Z_{>p}, i\neq j: z_i=z_j=\ell$.
\end{enumerate}
On $Z'$, we define $T'\colon Z'\to Z'$ by shifting $r$'s to the right, $\ell$'s to the left and letting $r$'s and $\ell$'s annihilate each other upon collision, i.e., $T'((z^\Rrm,r,\ell,z^\Lrm)) = (z^\Rrm,z^\Lrm)$ and $T'((z^\Rrm,r,0,\ell,z^\Lrm)) = (z^\Rrm,0,z^\Lrm)$ (here $z^\Rrm, z^\Lrm$ do not have $r$, $\ell$ at the right- or leftmost position, respectively). We define a map $U\colon Z\to Z'$ by
\begin{equation}
 (U(x))_i \coloneqq \begin{cases}
r & \textnormal{if }x_i=e+r\textnormal{ and }x_{i+1}\in\{e+r-1,e+r\},\\
\ell & \textnormal{if }x_i=e+r\textnormal{ and }x_{i-1}\in\{e+r-1,e+r\},\\
0 & \textnormal{otherwise}.\end{cases}
\end{equation}
which is surjective and satisfies $U\circ T|_{Z}=T'\circ U$. From a topological viewpoint, $Z'$ equipped with the cylinder topology renders $U$ continuous, i.e., $(Z',T')$ is a topological factor of $(Z,T|_{Z})$ so that $h(Z',T')\leqslant h(Z,T|_{Z})$. 

However, $U$ is not a bijection since $(0,e+r,0)$ (the last stage of an odd annihilation) is mapped to $(0,0,0)$ as is $(0,0,0)$ itself; the issue is that $e+r$ should be mapped to $r$ and $\ell$ simultaneously for consistency with $T$. Hence, for $z\in Z'$ with a block $(r,0_{2\CardA+k+1},\ell)$ (note there can be at most one such block) we have $\# U^{-1}(z) = k$, and infintely many preimages occur for $z$ with semi-infinite zero block. Nevertheless, $U$ has a unique inverse except in case of a block $(0,0,0)$.

We next follow \cite{DS91} in order to compute $h(Z',T')=2\ln\rho_{e+r}$ and define 
\begin{equation}
c_q(z'_{[k_1,k_2]})\coloneqq \#\{z'(k)=q: k\in [k_1,k_2]\}
\end{equation}
to count the number of symbols $q$ in the block $z'_{[k_1,k_2]}$ and
\begin{equation}
\gamma_{g,n}(q)\coloneqq \#\{z'_{[0,n-1]}:z'\in Z',c_q(z'_{[0,n-1]})=g\}
\end{equation}
to be the number of ways of putting down $g$ symbols $q$ on an integer interval of length $n$ with at least $e+r$ zeros inbetween. 
Moreover, let
\begin{equation}
\gamma_n(q)\coloneqq \sum_{g=0}^n \gamma_{g,n}(q),
\end{equation}
which is the number of ways of putting any number of symbols $q$ on an integer interval of length $n$ with at least $e+r$ zeros inbetween. This is the number of allowed words of length $n$ of the pure left or right subshifts, hence (cf. \cite{lind1995introduction})
\begin{equation}
\limsup_{n\to\infty}\frac{\gamma_n(\ell)}{n}=\ln\rho_{e+r}.
\end{equation}
where $\rho_{e+r}$ is the largest eigenvalue of the matrix $A$, cf. Figure~\ref{f:trans}.
 Finally, for a space-time window of symbols, $W_{m,n}\coloneqq \{0,r,l\}^{[-m,m]\times [0,n-1]}$, let 
\begin{equation}
\Gamma_{n,m}\coloneqq \textnormal{card}\{W_{m,n}: \exists z'\in Z': W_{m,n}=(z',T'z',\ldots,(T')^{n-1}z')_{[-m,m]}\}
\end{equation}
where the block formation is meant row-wise. This is the number of space-time windows $W_{m,n}$ that can be extended in space to be in an orbit of $T'$ on $Z'$.
The topological entropy of $h(Z',T')$ can then equivalently be defined as (cf. \cite{DS91,DU16})
\begin{equation}
h(Z',T')\coloneqq \sup_m\limsup_{n\to\infty}\frac{1}{n}\ln\Gamma_{m,n}.
\end{equation}

We first show that $h(Z',T')\geqslant 2\ln\rho_{e+r}$. To this end, we construct a sufficiently rich set of initial data in $Z'$ to generate different $W_{m,n}$. Let $t>0$ be some integer and $n=t(m+e+r)$. 

Consider the blocks $z'_{[-n,n]}$ for initial data of the form $z'=(z^\Rrm,0,z^\Lrm)\in Z'$, where $z^\Rrm$ is a semi-infinite configuration of $r$'s and $0$'s with at least $e+r$ $0$'s between two $r$'s and similarly for $z^\Lrm$.  Let $h_j$ for $\pm j=1,\ldots,t$ be blocks of length $m$ and consider
\begin{align*}
z^\Lrm_{[1,n]} &=(h_1,0_{e+r},h_2,0_{e+r}\ldots,h_t,0_{e+r}),\\
z^\Rrm_{[-n,-1]} &=(0_{e+r},h_{-t},0_{e+r},h_{-t+1},\ldots,h_{-2},0_{e+r},h_{-1})
\end{align*}
with $c_l(h_j)=c_r(h_{-j})$ for each $j=1,\ldots,t$. By construction, different such initial data differ at some point in $W_{m,n}$ and hence the number of these initial configuration is a lower bound for the total number of different space-time windows. Since for each $j$ we can independently assign symbols in $h_{\pm j}$, for each $j$ the number these pairs with $g$ non-zero symbols is $\gamma_{g,m}(r)\cdot\gamma_{g,m}(\ell)=\gamma_{g,m}(r)^2$. Since each subblock can be independently assigned symbols the total number of such initial data is
\begin{equation}
\left(\sum_{g=0}^m\gamma_{g,m}(\ell)^2\right)^t
\end{equation}
so that
\begin{align}
 h(Z',T')\geqslant\frac{1}{t(m+2)}\ln\left(\sum_{g=0}^m\gamma_{g,m}(\ell)^2\right)^t=\frac{1}{m+2}\sum_{g=0}^m\gamma_{g,m}(\ell)^2.
\end{align}
Finally, we use that for any $\varepsilon>0$ there exists an $m$ such that $\frac{1}{m+2}\sum_{g=0}^m\gamma_{g,m}(\ell)^2\geqslant\ln\rho_{e+r}-\varepsilon$ and thus $h(Z',T')\geqslant 2\ln\rho_{e+r}-\varepsilon$ for any $\varepsilon >0$ so that $h(Z',T')\geqslant 2\ln\rho_{e+r}$.

In order to prove $h(Z',T')\leqslant 2\ln\rho_{e+r}$, note that $\Gamma_{m,n}$ is at most the number of initial data $z'\in Z'$  for which a change in $z'$ has a chance to result in a change in $W_{m,n}$. Since the speed of propagation in one, it follows that only the blocks $z'_{[-m-n+1,m+n-1]}$ can have an impact on $W_{m,n}$. Moreover, any $\ell\in z'_{[-m-n+1,-m]}$ has no impact on $W_{m,n}$ since it is moving to the left. Likewise, any $r$ in $z'_{[m,m+n-1]}$ has no impact on $W_{m,n}$. Hence $\Gamma_{m,n}$ is less than or equal the number $N$ of configurations $z'\in Z'$ with $c_{\ell}(z'_{[-m-n+1,-m]})=0=c_r(z'_{[m,m+n-1]})$. Due to these restrictions on $z'_{[-m-n+1,m]}$ and $z'_{[-m,m+n-1]}$, we have that $N\leqslant\gamma_{2m+n}(r)\gamma_{2m+n}(\ell)=\gamma_{2m+n}(\ell)^2$, i.e.
\begin{equation}
h(Z',T')\leqslant\sup_m\limsup_{n\to\infty}\frac{1}{n}\ln \gamma_{2m+n}(\ell)^2=2\ln\rho_{e+r}.
\end{equation}

Since $U$ is not finite-to-one Bowen's inequality does not directly imply $h(T,Z)\leqslant 2\ln\rho_{e+r}$ as the fiber entropy is not a priori zero (cf. \cite{bowen1971entropy}). Since $e+r$ is mapped to $0$ under $T$, the previous counting misses precisely the options to replace $0$ with $e+r$ within a zero block in $W_{n,m}$. However, the preimage in $Z$ of $(0,e+r,0)$ under $T$ is $(e+r,e+r-1,e+r)$ on which $U$ is bijective. Hence, the only options for $W_{m,n}$ that have not been accounted for in $\Gamma_{m,n}$ concern the bottom row of $W_{m,n}$. Since these are at most $2m$, this extra contribution vanishes in the limit as $\ln(\gamma_{2m+n}(\ell)^2+2m)\leqslant \ln(\gamma_{2m+n}(\ell)^2) + \ln(2m)$, i.e., it carries zero entropy.
This concludes the proof.
\end{proof}

\begin{remark}\label{r:entropy}
In this proof the largest growth rate of different space-time-windows, and thus the topological entropy, stems from counting enduring annihilations, i.e., elements in the set $Z_\infty$. In this sense the topological entropy is generated by $Z_\infty$.
\end{remark}

Let us study what happens with the topological entropy as $c=e+r$ increases. Despite the increasing number of states, the system's complexity measured by the entropy actually decreases since the rigid reaction dynamics of incremental steps takes more time. 

\begin{lemma}~\label{l:asy}
The largest positive root  $\rho_c$  of the polynomial  $f_{c}(\lambda)\coloneqq \lambda^{c+1}-\lambda^{c}-1$ is strictly greater than one and  we have $ \rho_c-1 \sim {(\ln c)}/{c}$  as  $c\to\infty$.
\end{lemma}
\begin{proof}
By Descartes' rule of signs, the polynomial $f_c$ has exactly one positive (simple) root $\rho_{c}$. Moreover, $\rho_c>1$ since $f_c(\lambda)<0$ for $\lambda\in [0,1]$. Since $\left(1+\frac{\ln c}{c}\right)^c=e^{c\ln\left(1+\frac{\ln c}{c}\right)}\sim e^{\ln c}=c$ as $c\to\infty$ we have
\begin{equation*}
f_{c}\left(1+\frac{\ln(c)}{c}\right)=\left(1+\frac{\ln(c)}{c}\right)^{c}\frac{\ln(c)}{c}-1\sim\ln(c)-1>0,~c\to\infty.
\end{equation*}
Consequently $\rho_{c}\in\left(1,1+(\ln(c))/{c}\right)$ 
 for large values of  $c$ and $\lim_{c\to\infty}\rho_{c}=1$.

Now, since $f_c(1+\xi_c)=0$  if and only if $(1+\xi_c)^c\xi_c=e^{c\ln(1+\xi_c)}\xi_c=1$ 
and  $e^{c\ln(1+\xi_c)}\sim e^{c\xi_c}$, we get $c\xi_c e^{c\xi_c}\sim c$ as $c\to\infty$. Applying the Lambert function $W$ and using its large argument approximation, we find $c\xi_c\sim W(c)\sim\ln(c)$ as $c\to\infty$ which concludes the proof.
\end{proof}

However, viewing $e+r$ as a discretisation it is natural to rescale time accordingly by $c$ and this generates  even infinite entropy.
\begin{cor}\label{Cor:scaling}
Let $c=e+r$. The topological entropy of $T$ on $Z$ for a time-rescaling $\alpha(c)\in\N$ satisfies $h(Z,T^{\alpha(c)})=2\ln(\rho_c^{\alpha(c)})\sim 2\alpha(c)\frac{\ln c}{c}$ as $c\to\infty$. In particular, for $\alpha(c)=c$, the topological entropy  asymptotically doubles the topological entropy of the full $c$-shift.
\end{cor}
\begin{proof}
This is a direct consequence of Lemma~\ref{l:asy}.
\end{proof}
\begin{remark}
The restriction on $\Omega\setminus Z$ scales differently as the cardinality of excited and refractory states increases, cf. Remark~\ref{rem:asy}.
\end{remark}
%%%%%%%%%%%%%
\subsection{Waiting times, coherent structures and chaos}\label{s:coh}

Since the dynamics on $Z_\infty$ consists entirely of pulses moving towards each other until collision, it is natural to encode this in terms of waiting and annihilation times. Each $x\in Z_\infty$ away from pulse annihilation and collision has the form 
\[
(\ldots, 0_{k_j}, \P^\Rrm, 0_{k_{j+1}},\ldots,\P^\Rrm,0_{k_0},\P^\Lrm,0_{k_1},\P^\Lrm,\ldots)
\]
for a sequence of \emph{waiting times} $(k_j)_{j\in \Z}\subset \N$ between the end of an annihilation and the next collision. Recall that a collision occurs precisely when $\Isp(x)=[p_-,p_+]$, $1\leqslant p_+-p_-\leqslant 2$ and $x_{p_-}=x_{p_+}=1$; the end of the collision is reached one iteration step after $x_{p_-}=x_{p_+}=e+r$.

Clearly, $k_0+1=\#\Isp(x)$ and if $k_0\geqslant 1$ then the position of the initial collision, $c_0\in \Z$ under the dynamics of $T$ is at the lattice site of the (left-)center $c_0=\lfloor(p_-+p_+)/2\rfloor$ of $\Isp(x)=[p_-,p_+]$, and this happens at time $t_0=\lfloor k_0/2\rfloor$. Otherwise, the initial condition lies in an annihilation event in which case $c_0$ is as in the previous case, but $t_0=1-x_{p_-}$ should be considered negative. More generally, the lattice site and time of the $n$-th collision, $n\geqslant 1$, can be computed from $k_j$ for $j=0,\ldots, n$ recursively as
\begin{align*}
c_n = c_{n-1} + \left\lfloor\frac{k_{n}-k_{-n}}{2}\right\rfloor, \qquad
t_n = t_{n-1} + e+r + \left\lfloor\frac{k_n+k_{-n}}{2}\right\rfloor
\end{align*}
which can be readily written as explicit summations. In particular, the collision sites remain constant as long as $k_j=k_{-j}$ for an interval $j\in [j_-,j_+]\subset \N$. 

We point out an analogy to so-called coherent structures found in nonlinear waves in partial differential equations, in particular the complex Ginzburg-Landau equation. In all cases we require collision times to be equidistant, i.e., $t_{n+1}-t_n = e+r + \tau, \tau\in \N$, which means $k_{n+1} +  k_{-n-1}=\tau$. We may view $x$ as a \emph{sink defect}, if the collision sites move with constant speed $s$, i.e.,
\[
\frac{c_{n+1}-c_n}{t_{n+1}-t_n} = s \quad \Leftrightarrow c_{n+1} = c_n + s (e+r+\tau).
\] 
However, since there is no dispersion of waves and thus no variation of group velocity, using this terminology is a slight abuse of language.

Manipulating on $c_n, t_n$ we can create a rich set of solutions, whose complexity can be measured by the topological entropy discussed later. One may interpret this as the entropy of the admissible sequences of pairs $(c_n,t_n)_{n\in \N}\subset \Z\times \N$, though we do not pursue this viewpoint here.

\medskip
Rather we use the waiting time coding to investigate the sensitivity of the dynamics. Note that also any $x\in Z$ has `waiting times', i.e., lengths of zero intervals between local pulses forming a sequence $(k_j)_{j\in J}$, where either $J=\Z$ or $J=[-\infty, j_+]$ with $j_+\geqslant 0$ and $k_{j_+}=\infty$, or $j=[j_-,\infty]$, $j_-\leqslant 0$, $k_{j_-}=\infty$.

Moreoever, one can use this coding to argue that the basin of attraction of $0_{\infty}$ is dense in $X$.  

\begin{cor}\label{cor:densebasin}
$\Gamma_-(0_{\infty})\coloneqq \bigcup_{n\in\N}T^{-n}(0_{\infty})$ is dense in $X$. In particular, $M\cap \Gamma_-(0_{\infty})$ is dense in $M\in\{Y,Z\}$ with respect to the subspace topology. Analogously, any $x\in Z$ has a dense basin of attraction.
\end{cor}
\begin{proof}
It suffices to prove $[a_0,\ldots,a_n]_m\cap\Gamma_-(0_{\infty})\neq\emptyset$ for all cylinder sets. To this end, let $x_{[m,m+n]}=(a_0,\ldots,a_n)$. Choose $x_{[-\infty,m-1]}\in\SRneg$ and $x_{[m+n+1,+\infty]}\in\SLpos$ with finitely many waiting times $(k_j)_{j\in J}$ and corresponding pairs $(c_n,t_n)$, respectively, such that there exist counter propagating local pulses which annihilate pairwise and $x_{i}=0, \lvert i\rvert\geqslant j$, for sufficiently large $j\in\Z$. Then, $T^{t_N}(x)=0_{\infty}$ for some maximal collision time $t_N\in\N$.
\end{proof}

\begin{prop}\label{prop:ZinNW}
The subsystem $(T|_Z,Z)$ is chaotic in the sense of Devaney: (i) periodic orbits are dense, (ii) it is topologically  transitive, and (iii) sensitive with respect to initial conditions. In particular, $Z$ is contained in the non-wandering set $Z\subset\Omega$.
\end{prop}

\begin{proof}
(i) It is well known that the statements holds for both invariant subsystems $\SR$ and $\SL$.
Without loss of generality, let $x\in Z\setminus(\SR\cup\SL)$ with $x=(x^\Rrm,0_\ell,x^\Lrm), \ell\in \N$, and $\Isp=[p_-,p_+]$, $x_{[p_-,p_+]}=(1,0_\ell)$ for suitable $(x^\Rrm, x^\Lrm)\in\SLR$. Consider an arbitrary neighbourhood  $U$ of $x$. Due to the structure of $Z$ and the cylinder topology there are $k_\pm\in \pm\N$ such that the following holds: $\Isp\subset [k_-+1,k_+-1]$, $x_{k_\pm}=0$ and for any $(\tx^\Rrm, \tx^\Lrm)\in\SLR$ we have $\tx=(\tx^\Rrm,x_{[k_-,k_+]},\tx^\Lrm)\in U$ (with positioning so that $\tx_{k_+}=x_{k_+}$).

In order to make $\tx$ time periodic we distinguish even and odd $\ell$. In case of even $\ell$ let $p_m\coloneqq (p_+-p_--1)/2$ and choose $\tx^\Rrm$ as the leftward periodic extension of $x_{[k_-,p_m]}$ and $\tx^\Lrm$ the rightward periodic extension of $x_{[p_m+1,k_+]}$. The odd case is analogous.

(ii) Let $U,V\subset Z$ be any open nonempty $f\coloneqq T|_{Z}$-invariant sets. For an index set $K$, $k\in K$ and fixed $a_{0_k},\ldots,a_{m_k}\in\calA, n_k\in\Z$, let $C_{k}^Z\coloneqq [a_{0_k},\ldots,a_{m_k}]_{n_k}\cap Z$.
With this notation, $U$ and $V$ are of the form $U=\bigcup_{i\in I}C_i^Z$ and $V=\bigcup_{j\in J}C_j^Z$. Since $U,V$ are $f$-invariant, $f^n(C_i^Z)\subseteq U$ and $f^n(C_j^Z)\subseteq V$ for all $i\in I, j\in J$ and $n\in\mathbb{N}$. By Corollary~\ref{cor:densebasin}, there exists some $N\in\N$ such that $0_{\infty}\in f^N(C_i^Z)\cap f^N(C_j^Z)\subseteq U\cap V$. By Corollary 1.4.3 \cite{katok1997introduction}, $f$ is topologically transitive.

(iii) This already follows by (i) and (ii), cf. \cite{WANG2009803}, but can also be shown explicitly. To this end let $x\in Z\setminus(\SR\cup\SL)$ with has waiting times $(k_j)_{j\in J}$. Let $U$ be an arbitrary neighbourhood of $x$ and $j_0$ such that any change of finite $k_j$ with $|j|\geqslant j_0$ gives a point in $U$. In particular, unless  already $k_{j_0}=\infty$ or $k_{-j_0}=\infty$ we choose $x'$ with waiting times $k'_j=k_j$ for $|j|\leqslant j_0$ (and $j\in J$) such that this holds so that $T^m(x')\in \SR\cup\SL$ for which sensitivity is straightforward.
\end{proof}

%%%%%%%%%%%%%%
\subsection{Stationary dislocations}\label{s:disloc}

We introduce another class of configurations $x\in X$ related to $Z$, which possess \emph{dislocations} $s(x_j,x_{j+1})>1$ so that $x\notin Z$. Consider first $x\in X$ with $x=(x^\Rrm~^p|x^\Lrm)$, $x^\Rrm\in\SRneg, x^\Lrm\in\SLpos$, with one dislocation at $p$ that is then naturally a generalised separating position. An example is the following orbit, 
\begin{align*}
x =(\ldots, 3,2~^p&|0,1,2,\ldots)\\
\vdots\\
T^{\CardA-2}(x) =(\ldots,a,0~^p&|e+r-1,e+r,0,\ldots)\\
T^{\CardA-1}(x) =(\ldots,a'~^p&|e+r,0,b,\ldots)\\
T^{\CardA}(x) =(\ldots,a''~^p&|0,b',\ldots).
\end{align*}
For $a=0$ we have $a'=0$ and possibly $a''=1$ and since $b'\in\{0,1\}$ the next pulse collision would be as in $Z$. Hence, in order to maintain a difference from $Z$ requires to maintain a dislocation at the separating position $p$. The simplest option is $a=b=1$ so that $a'=1$ and $a''=2$ so that the same dislocation as in $x$ occurs. Let $x^{\Rrm*_\pm}$ and $x^{\Lrm*_\pm}$ be the highest frequency pulse sequence with one $0$ between local pulses left-infinite (`$-$') or right-infinite (`$+$'). Choosing $x^\Rrm=x^{\Rrm*_-}, x^\Lrm=x^{\Lrm*_+}$ creates a periodic orbit with period $\CardA$ and constant separating position. This construction generalises to any choice of $x_p, x_{p+1}$, and we refer to such a periodic solution, as well as the following ones, as \emph{periodic stationary dislocation}.

In fact, we may place dislocations next to each other as long as one neighbour at each $0$ lies in $E$. This yields an interval of dislocations of arbitrary width between $x^{\Rrm*_-}$ and $x^{\Lrm*_+}$, and the width and position of the dislocations for such a solution remains constant and the solutions has period $\CardA$, cf. Figure~\ref{fig:dislocations} (d). Moreover, for $s(x_j,x_{j+1})\in E$, excitations are transported from right to left so that we may take an interval of such dislocation between $x^{\Rrm*_-}$ the left and $x^{\Rrm*_+}$ thus creating a kind of so-called `transmission defect', cf. \textsection\ref{s:coh}.

\medskip
Notably, if we select to the left $x^\Rrm$ or to the right $x^\Lrm$ with a longer waiting times, then the dynamics will map into $Z$ after finitely many steps, except for the following type of stationary dislocations.

For $e+r\geqslant 4, r>e+1$ there is a class of stationary dislocations with arbitrarily long period, even aperiodic, having a unique separating position. Consider the $(x^\Lrm, x^\Rrm)\in\SLR$ with pulse distances periodically alternating between one and two, cf. Figure~\ref{fig:dislocations} (a): 
\begin{align*}
x^\Lrm &=(0,0,1,2,\ldots,e+r,0,1,2,\ldots,e+r,0,0,1,2,\ldots,e+r,0,1\ldots),\\
x^\Rrm &=(\ldots,1,0,e+r,e+r-1,\ldots, 1,0,0, e+r, e+r-1, \ldots,1,0, e+r,\ldots,e+1)
\end{align*}
and set $x\coloneqq (x^\Rrm~^p| x^\Lrm)$. Following the dynamics of this initial state gives at $j=e+2r$ that $T^j(x)_{[p,p+1]}=(e+r,r-2)$ and $T^{j+1}(x)_{[p,p+1]}=(0,r-1)$. The assumption $r-1\in R$ gives $T^{j+2}(x)_{[p,p+1]}=(0,r)$, which is consistent with the zero block in $x^\Rrm$ and thus we obtain a periodic solution of minimal period $2\CardA+1$. For more generality, see section~\ref{s:entropy}.

In this fashion we can map the full 0-1-shift onto such solutions thus creating another chaotic invariant subset. However, as shown in section~\ref{s:entropy} this has smaller entropy than $Z$ and the non-wandering set is formed by $Z$ together with the configurations of this subsection.
\begin{figure}[t] \centering
\begin{tikzpicture}
\node[above right] (img) at (0,0) {\includegraphics[scale=0.5]{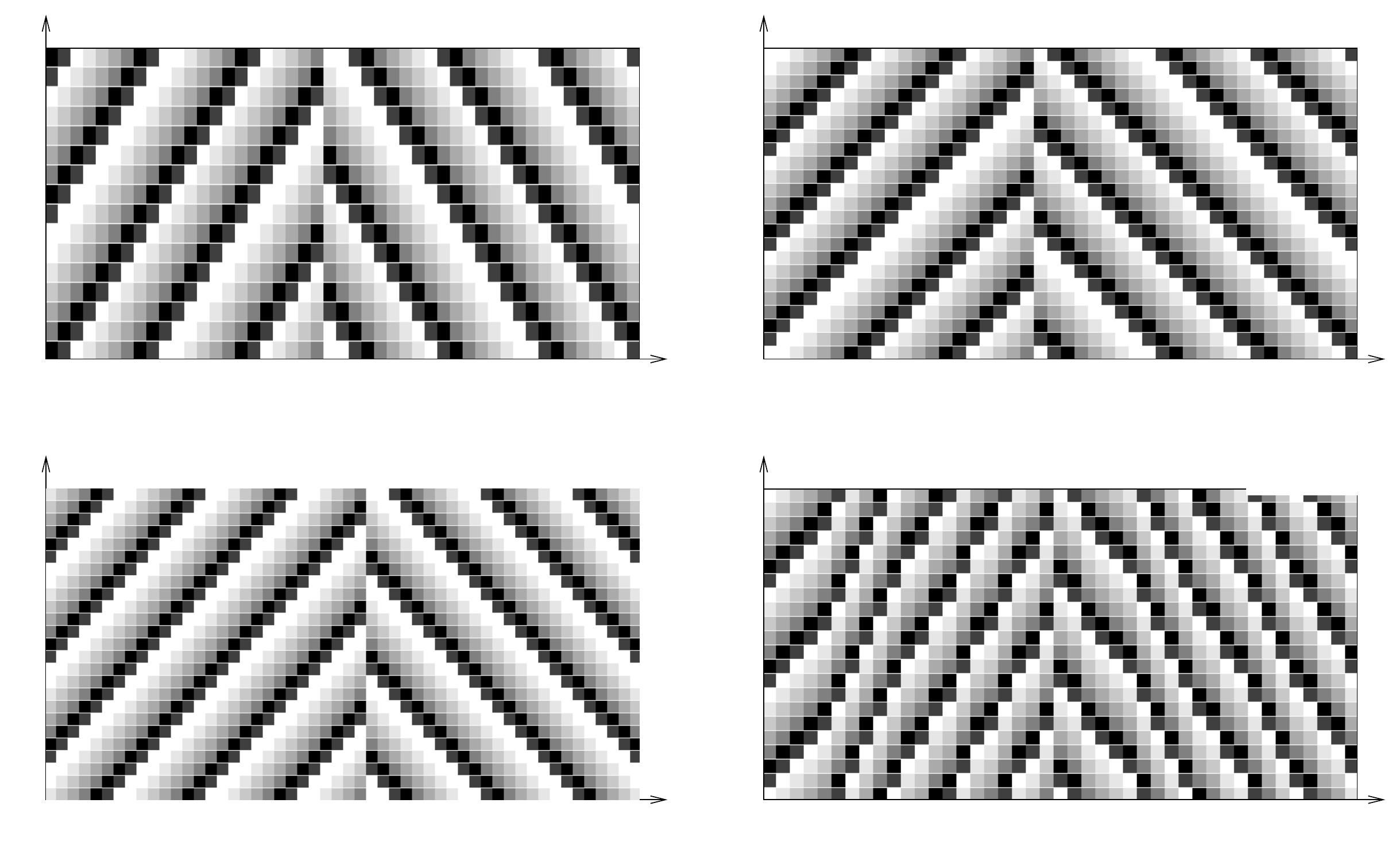}};
\node at (220pt,10pt) {$x\in\OmConst$};
\node at (220pt,122pt) {$x\in\OmVar$};
\node at (39pt,10pt) {$x\in\OmVar$};
\node at (39pt,122pt) {$x\in\OmVar$};
\node at (335pt,10pt) {space};
\node at (335pt,122pt) {space};
\node at (153pt,10pt) {space};
\node at (153pt,122pt) {space};
\node[rotate=90] at (5pt,202pt) {time};
\node[rotate=90] at (186pt,202pt) {time};
\node[rotate=90] at (186pt,91pt) {time};
\node[rotate=90] at (5pt,91pt) {time};
\node at (270pt,10pt) {(d)};
\node at (270pt,122pt) {(b)};
\node at (89pt,10pt) {(c)};
\node at (89pt,122pt) {(a)};
\end{tikzpicture}
\caption{Snapshots of configurations with (spatially) stationary dislocations (again $e=2$ and $r=4$ with the same coloring as for Figure 2). The panels (a)-(c) show configurations in $\OmVar$ with unique separating position (dislocation). Panel (d) illustrates an element of $\OmConst$ with an interval of separating positions (dislocations), cf. \textsection\ref{s:constantstepsize} and \textsection\ref{s:varyingstepsize} for details. Note that, in time, the dislocations are either periodic as suggested in (a) and (c) or aperiodic as in (b).}
\label{fig:dislocations}
\end{figure}
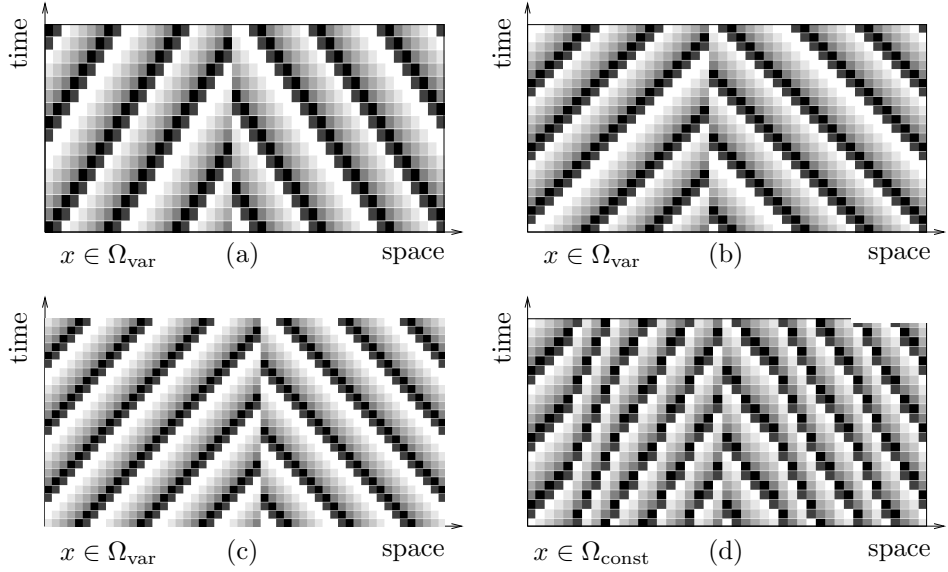

%%%%%%%%%%%%%
\subsection{Skew-product structure}\label{sec:skew}

In Proposition~\ref{p:entropy} we showed $h(Z,T|_Z)=2\ln\rho_{e+r}$, and that it is combinatorially generated by $Z_\infty$, cf.\ Remark~\ref{r:entropy}. This value is exactly the sum of the topological entropies of the left and right subshifts (cf.\ Lemma~\ref{l:invariants}(i)) for which $h(\SR,T|_{\SR})=h(\SL,T|_{\SL})=\ln \rho_{e+r}$. Thus, the well-known product formula $h(\sigL\times\sigR)=h(\sigL)+h(\sigR)$ suggests a product structure of the dynamics of $T$ on $Z$. Indeed, if for some $x\in Z$ two pulses approach each other under $T$, the dynamics is exactly a combination of left and right shifts. However, once two pulses have annihilated, the position of the next annihilation can be arbitrarily far away, which causes problems to find a conjugacy of $T$ on $Z$ to some form of product system.

Here we present a topological conjugacy to a skew-product systems for the restriction to $Z_\infty$. Note that restricting on the dense subset $Z_{\infty}\subset Z$ makes no difference in terms of the Bowen-Dinaburg entropy.
\begin{prop}
$h(Z,T|_{Z})=h(Z_{\infty},T|_{Z_{\infty}})$
\end{prop}
\begin{proof}
Since $Z$ is a totally bounded metric space and $Z_{\infty}\subset Z$ is dense by Lemma~\ref{l:invariants} (iv), this is a direct consequence of \cite{hasselblatt2005topological}, Corollary 4.
\end{proof}

Consider the left and right subshifts  placed at fixed positions,
\begin{align*}
\Sigma_{\Lrm}^+&\coloneqq \left\{x=(x_i)_{i\in\mathbb{Z}_{>0}}\in\left\{0,1,2\right\}^{\mathbb{Z}_{>0}}: a_{x_i,x_{i-1}}=1\right\},\\
\Sigma_{\Rrm}^- &\coloneqq \left\{x=(x_i)_{i\in\mathbb{Z}_{\leq 0}}\in\left\{0,1,2\right\}^{\mathbb{Z}_{\leq 0}}: a_{x_i,x_{i+1}}=1\right\},
\end{align*}
on which the standard right shift $\sigRneg$ and left shift $\sigLpos$ are defined with pseudo-inverses for $m<0$ given by
\begin{equation*}
\sigRneg^m((\ldots,x_{-1},x_0))\coloneqq (\ldots,x_{-1},x_0,0_m),\qquad\sigLpos^m((x_0,x_1,\ldots))\coloneqq (0_m,x_0,x_1,\ldots).
\end{equation*}
Each configuration $x\in Z_{\infty}$ can be written as $x=(x^\Rrm~^p|~x^\Lrm)$ with $x^\Rrm\in\Sigma_{\Rrm}^-\setminus\{0_\infty^-\}$ and $x^\Lrm\in\Sigma_{\Lrm}^+\setminus\{0_\infty^+\}$. Here $p\in\Z$ is some separating position and hence a priori not unique. Since separating positions form an interval $\Isp(x) = [p_-,p_+]$, with $p_\pm=p_\pm(x)$, one option for a unique choice is the \textit{middle separating position}
\begin{equation*}
\pmid(x)\coloneqq \left\lfloor \frac{p_+(x)+p_-(x)}{2}\right\rfloor.
\end{equation*}

For $x=(x^\Rrm~^p|~x^\Lrm)\in Z_\infty$ we define the \textit{adaption of $p$ to the middle separating position of $Tx$} by $a\colon Z_\infty\to\Z, a(x):=\pmid(Tx)-p$.
\begin{comment}
\begin{equation*}
a(x)\coloneqq \begin{cases}p_{\textnormal{mid}}(Tx)-p\in\mathbb{Z}, & x=(x^\Rrm~^p|~x^\Lrm), p\in\mathbb{Z}, x^\Rrm\neq 0_\infty^-, x^\Lrm\neq 0_\infty^+, T(x)\notin\SR\cup\SL\\
0, & x=(x^\Rrm~^p|~x^\Lrm), p\in\mathbb{Z}, x^\Rrm\neq 0_\infty^-, x^\Lrm\neq 0_\infty^+, T(x)=0_\infty\\
+1, & x=(x^\Rrm~^p|~x^\Lrm), p\in\mathbb{Z}, x^\Rrm\neq 0_\infty^-, x^\Lrm\neq 0_\infty^+, T(x)\in\SR\setminus\{0_\infty\}\\
-1, & x=(x^\Rrm~^p|~x^\Lrm), p\in\mathbb{Z}, x^\Rrm\neq 0_\infty^-, x^\Lrm\neq 0_\infty^+, T(x)\in\SL\setminus\{0_\infty\}\\
0, & x\in \SR\cup\SL
\end{cases}.
\end{equation*}
\end{comment}

While the adaption in $Z_\infty$ is always bounded there is no a priori bound after pulse collisions. As long as $\pmid$ is constant during iteration of $T$, i.e. $a(T^ix)=0$ for an interval in $\N$, the dynamics of $T$ is a product of left- and right-shift centered at position $p=\pmid(x)$, 
\[
T(x)= \left(\sigRneg(x)~^p\left|~\sigLpos(x)\right.\right).
\]
This occurs in an invariant subset of $\Omega$ (cf. proof of Theorem~\ref{maintheorem}), but otherwise the description requires some technicalities.

We look for a set $\calZ$ and a map $F\colon\calZ\to\calZ$, which is in essence a product of left and right shift and admits an appropriate continuous conjugation $H\colon Z_\infty\to \calZ$, i.e., $H\circ T|_{Z_\infty}=F\circ H$ obeys a commutative diagram
\begin{equation*}
\begin{CD}
Z_\infty     @>T>>  Z_\infty\\
@VHVV        @VVHV\\
\calZ     @>F>>  \calZ.
\end{CD}
\end{equation*}
To this end, let $\Sigma_-^+\coloneqq \Sigma_{\Rrm}^-\times \Sigma_{\Lrm}^+\setminus \Sigma_*$, where $(x^\Rrm,x^\Lrm)\in \Sigma_*$ if $x=(x^\Rrm~^0|~x^\Lrm)$ does not lie in $Z_\infty$, which is precisely the case if $x\in\SLR$ has a dislocation at the block $x_{[0,1]}=(x^\Rrm_0,x^\Lrm_1)$ or if $x\in \SLR$ with $x_{[-\infty,k]}=0_\infty^-$ or $x_{[k',\infty]}=0_\infty^+$ for some $k,k'\in\Z$.

On the product space $\calZ:=\Sigma_-^+\times\Z$ define the skew-product map 
\begin{equation*}
F\colon\calZ\to\calZ, F(z)=(f(x),g_x(p)),\quad z=(x,p)\in\calZ,
\end{equation*}
with base function $f\colon\Sigma_-^+\to\Sigma_-^+$,
\begin{equation*}
    f(x)=\left(\sigRneg^{-\alpha+1}(x^\Rrm), \sigLpos^{\alpha+1}(x^\Lrm)\right),\quad x=(x^\Rrm,x^\Lrm)\in \Sigma_-^+
\end{equation*}
where $\alpha=a(x^\Rrm~^0|~x^\Lrm)$, $\sigRneg^{-\alpha+1}=\sigRneg^{-\alpha}\circ\sigRneg$ and $\sigLpos^{\alpha+1}=\sigLpos^\alpha\circ\sigLpos$. 

For given $x=(x^\Rrm,x^\Lrm)\in\Sigma_-^+$, the fibre function $g_x\colon\Z\to\Z$ is defined as
\begin{equation*}
    g_x(p)=p+a(x^\Rrm~^0|~x^\Lrm).
\end{equation*}
Finally, the bijection $H\colon Z_\infty\to\calZ$ and its inverse $H^{-1}$ are given by
\begin{equation*}
    H(x)=(x^\Rrm,x^\Lrm,\pmid(x)),\quad H^{-1}((x^\Rrm,x^\Lrm,p))=(x^\Rrm~^p|~x^\Lrm).
\end{equation*}
Taking the cylinder topologies on $\Sigma_{\Rrm}^-, \Sigma_{\Lrm}^+$ gives the natural topology on $\calZ$ in this context. Each block $x_{[m,n]}, m,n\in\Z$, of $Z_\infty$ either is a block in $\SL, \SR$ or contains a separating position in $[m,n]$. In the latter case, it follows that the image under $H$ of a cylinder defined by this block gives a product of cylinders in $\Sigma_{\Rrm}^-, \Sigma_{\Lrm}^+$. In case a cylinder is defined by a block in $\SL$ or $\SR$ its image under $H$ is a  union of cylinders defined by extended blocks with a separating position, and therefore open in $\calZ$. Hence, $H$ is continuous and similarly we infer continuity of $H^{-1}$.

\bigskip
Though instructive, the conjugacy does not directly help to determine or sharply estimate the topological entropy. Indeed, since $Z_\infty$ is not compact, it would be desirable to identify a conjugacy for $Z=\overline{Z_\infty}$. However, it is unclear how to track positions that separate left and right shifts in a consistent manner, e.g.,  $0_\infty$, the bi-infinite zero sequence, does not have a canonical separating position. For elements in $\SR\cup\SL$ one may choose $\pm\infty$ as a separating position, but the pre-images of elements with a semi-infinite zero sequence are not unique. Attempts to consider quotient spaces cause additional difficulties to determine the topological entropy in the skew-product system. 

Disregarding $H$, we may extend $F$ to a closure of $\calZ$, which is canonical for $\Sigma_-^+$. Projecting away the position reduces to a coupled system of left- and right-shifts given by $f:\overline{\Sigma_-^+}\to \overline{\Sigma_-^+}$. By Bowen's inequality its topological entropy gives a lower estimate for that of $F$; if one can show that the fiber entropy vanishes, cf. \cite{kolyada1996topological,shimomura1986,canovas2013topological}, it would also be an upper estimate. However, one would still need to determine the topological entropy of $f$. 

Since the exponents $1-a(x)$ and $1+a(x)$ sum up to $2$, one would expect that $f$ behaves asymptotically as the uncoupled product $(\sigRneg,\sigLpos)$. Indeed, for a modification $\tilde f$, where $a(x)\in\{0,\pm 1\}$, it is possible to apply Bufetov's definition of topological entropy of free semigroup actions \cite{Bufetov1999,doi:10.1080/14689367.2017.1298724} and verify $h(\tilde f)=2\ln\rho_{e+r}$.  However, for $f$ the exponent $a$ has no a priori bound, and we are not aware of a treatment of this case elsewhere in the literature. Here we do not further pursue this alternative and thus rely on determining the topological entropy combinatorially as in the proof of Proposition~\ref{p:entropy}.

%%%%%%%%%%%%% NON-WANDERING SET %%%%%%%%%%%%%%%%%%
\section{The non-wandering set and the topological entropy}\label{s:entropy}
We now  turn our attention to $T$ on the entire space $X$. The aim of this section is to determine the non-wandering set which turns out to consist of the collision subsystem $Z$ and the stationary dislocations. This allows us to determine the topological entropy by inferring the upper estimate $h(X,T)\leqslant h(Z,T|_{Z})$.

%\subsection{Preliminaries and forbidden triples}
We first recall the definition of the non-wandering set \cite{katok1997introduction,walters2000introduction}.
\begin{definition}\label{nonwandering}
Let $T\colon X\to X$ be continuous. The set 
\begin{equation*}
\Omega\coloneqq \Omega(T)\coloneqq \{x\in X:\textnormal{ for every neighbourhood }U\textnormal{ of }x~\exists N\geqslant 1: T^{N}U\cap U\neq\emptyset\}
\end{equation*}
is called the non-wandering set of $T$. For compact $X$, one can equivalently assume that there exist arbitrarily large $N\in\N$ for which the intersection is nonzero \cite{katok1997introduction}.
\end{definition}

\begin{remark}\label{rem:cylsets}
For the product topology on $X=\calA^\Z$, a configuration $x\in X$ is non-wandering, $x\in\Omega$, if and only if for any cylinder set $[x_n,\ldots,x_{n'}]_n$ there exist arbitrary large $N\in\N$ such that $T^N([x_n,\ldots,x_{n'}]_n)\cap [x_n,\ldots,x_{n'}]_n\neq \emptyset$.
\end{remark}

For our setup, $\Omega$ is contained in the eventual image $Y\coloneqq \bigcap_{n\in\N}T^n(X)$, which can be shown to have shift space structure, cf.\ \cite{KRU}. Since we make use of $\Omega\subset Y$, we prove this inclusion here in lack of a reference. 
\begin{lemma}\label{lem:NWinY}
Let $X$ be a compact Hausdorff space with a continuous transformation $T\colon X\to X$. Then the non-wandering set is contained in the eventual image, i.e. $\Omega\subseteq Y$.
\end{lemma}
\begin{proof}
Let $n\in\mathbb{N}$ and $x\in U\coloneqq T^n(X)^C$. $U$ is an open neighborhood of $x$ with
\begin{equation*}
T^m(U)\subseteq T^n(X)=U^C\textnormal{ for all }m\geqslant n.
\end{equation*}
By \cite{walters2000introduction}, Theorem 5.7, $x$ is a wandering point. Hence $\Omega\subseteq\bigcap_{n\in\mathbb{N}}T^n(X)=Y$. 
\end{proof}
\begin{remark}\label{rem:propersubset}
(i) In the degenerate case $e=r=1$, the non-wandering set and the eventual image coincide, $Z=\Omega=Y$ \cite{DU16}.  (ii) One can shown that for $e\neq 1$ or $r\neq 1$, $\Omega$ is a proper subset of $Y$, $\Omega\subsetneq Y$, cf.\ \cite{KRU}.
\end{remark}

\subsection{Strategy for characterising $\Omega$}

The aim of this section is to characterise the non-wandering set. In this regard, Remark~\ref{rem:cylsets} provides an instruction: An element $x\in X$ is contained in $\Omega$ exactly if any finite block $x_{[n,n']}$ of $x$ can be restored infinitely often at the same positions under iterations of $T$. In particular, this requires the local dynamics of the blocks to be globally synchronised. 

Our strategy to reveal this matching of local and global dynamics is the following:   
\begin{enumerate}
\item[\textbf{I.}]\textbf{Local analysis.} We start with a local analysis of non-wandering points at each position $p\in \Z$ and show that for $x\in\Omega$ the trajectory $\{(T^m(x))_{[p,p+1]}: m\in\N_0\}$ of a $2$-block $x_{[p,p+1]}$ can be completely described in terms of the transitions of the associated step-size 
\begin{equation*}
s_p^m(x)\coloneqq s((T^m(x))_{p},(T^m(x))_{p+1}) = (T^m(x))_{p+1} - (T^m(x))_{p} \mod \CardA,
\end{equation*}
which decompose into equivalence classes (`communicating classes'), cf. Figures~\ref{fig:cc} and \ref{fig:cc1}.
\item[\textbf{II.}]\textbf{Global analysis.} We use the local step-size analysis to infer the spatial structure of $x\in\Omega$ by showing that the local dynamics of any $2$-block $x_{[p,p+1]}$ essentially determines the global spatial structure. 
\end{enumerate}
By this approach, we characterize $\Omega$ and the topological entropy $h(\Omega,T|_{\Omega})$. More specifically, we will identfy sets $\OmConst$ and $\OmVar$ that correspond to (periodic or aperiodic) configurations with stationary dislocations of certain constant or varying step-sizes. These form the complement of $Z$ within $\Omega$ and the main results of this section are the following.

\begin{theo}\label{NWeleqr}
The non-wandering set, $\Omega$, of the 1D-Greenberg-Hastings cellular automaton associated with $T:X\to X$ is the disjoint union of invariant sets
\begin{equation*}
\Omega=Z\uplus\OmConst\uplus\OmVar,
\end{equation*}
where $\OmVar=\emptyset$ if $e+1\geqslant r$  and $\OmVar=\emptyset$ otherwise.
\end{theo}

\begin{theo}\label{maintheorem}
The topological entropy, $h(X,T)$, of the 1D-Greenberg-Hastings cellular automaton with $e,r\in\N$ is given by $h(X,T)=h(Z,T|_{Z})=2\ln\rho_{e+r}$, where $\rho_{e+r}$ is the positive root of $x^{e+r+1}-x^{e+r}-1$. Moreover, $h(\OmConst,T|_{\OmConst})=0$ and $0<h(\OmVar,T|_{\OmVar})<2\ln\rho_{e+r}$ for $r>e+1$.
\end{theo}

\begin{remark}\label{r:sources}
These results imply that the recurrent structure for $e>r$ is somewhat simpler since $\OmVar=\emptyset$. However,  there are in fact new wave phenomena compared to $r\leqslant e$ in the eventual image: an interesting case are sources emitting pulses to the left and right antisynchronously, cf. Figure~\ref{fig:spiral}. These qualitatively share features of \textit{1D-spirals} observed experimentally in a quasi-1D chemical system and, in continuous models, were related to  localized periodic Turing states \cite{PhysRevLett.71.1272}.  Here the asymptotic state is $\CardA$-periodic and lies in $\OmConst$.
\end{remark}

\begin{figure}[h]\centering
\begin{tikzpicture}
\node[above right] (img) at (0,0) {\includegraphics[scale=0.3]{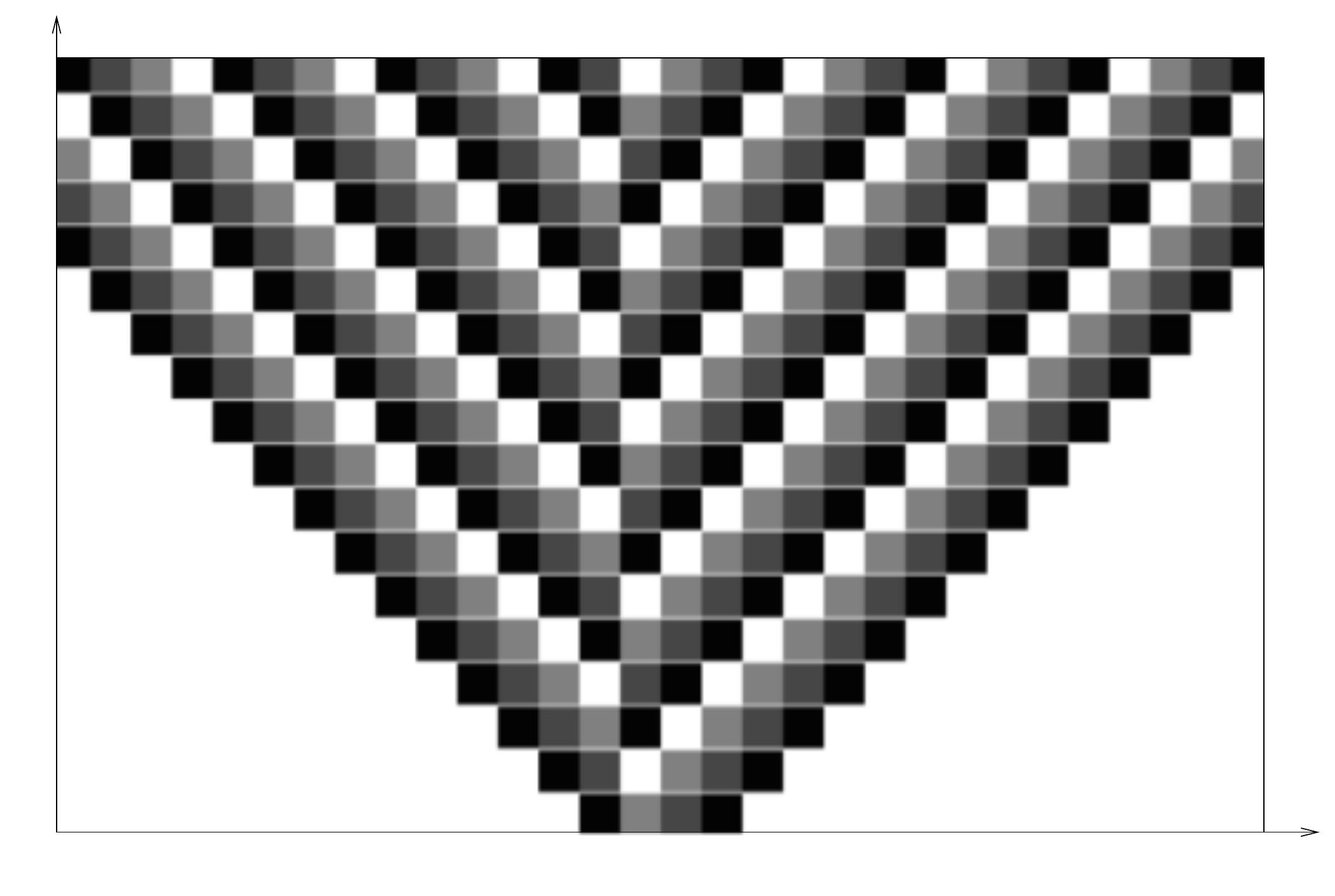}};
\node at (180pt,5pt) {space};
\node[rotate=90] at (3pt,110pt) {time};
\end{tikzpicture}
\caption{1D-spiral. $e=2, r=1$ ($0$ shown in white; $1,2,3$ in decreasing gray levels)} 
%\begin{comment}The initial configuration $x$ has no separating positions due to the source block $b=(0,1,2,3,1,0)$ and contains no forbidden blocks. Moreover, $x\in Y$ which can be readily seen by constructing a backward orbit of $x$ by repeating its block $b$ periodically in backward time with period $\CardA$ and embedding it such that no counter-propagating pulses in forward time disperse.\end{comment}
\label{fig:spiral}
\end{figure}

\begin{figure}
\begin{minipage}{0.8\textwidth}
\centering
\begin{tikzpicture}[scale=0.7, transform shape,shorten >=1pt,->]
  \tikzstyle{vertex}=[shape=circle,dashed,draw,minimum size=37pt,inner sep=0pt]
\node[vertex, fill=gray!60, opacity=0.5] (Ce+1) at (0,0){};
\node at (0,0){$C_{\pi}$};
\node[vertex] (Ce) at (2,1){$C_e$};
\node[vertex] (Cr+1) at (2,-1){$C_{r+1}$};
\node[vertex] (Cr+2) at (4,-1){$C_{r+2}$};
\node[vertex] (Ce-1) at (4,1){$C_{e-1}$};
\node[vertex,fill=gray!60,opacity=0.5] (C0) at (10,0) {};
\node at (10,0){$C_0$};
\node[vertex] (Ccdots1) at (6,1){$\cdots$};
\node[vertex] (Ccdots2) at (6,-1){$\cdots$};
\node[vertex] (C2) at (8,1){$C_2$};
\node[vertex] (Ce+r-1) at (8,-1){$C_{e+r-1}$};

\path (C2) edge[->] (C0);
\path (C2) edge[->, loop above] (C2);
\path (C0) edge[->, loop right] (C0);
\path (Ccdots1) edge[->, loop above] (Ccdots1);
\path (Ce-1) edge[->, loop above] (Ce-1);
\path (Ce) edge[->, loop above] (Ce);
\path (Ce+1) edge[->, loop left] (Ce+1);
\path (Cr+1) edge[->, loop below] (Cr+1);
\path (Cr+2) edge[->, loop below] (Cr+2);
\path (Ccdots2) edge[->, loop below] (Ccdots2);
\path (Ce+r-1) edge[->, loop below] (Ce+r-1);

\path (Ce+r-1) edge[->] (C0);
\path (Ccdots2) edge[->] (Ce+r-1);
\path (Cr+2) edge[->] (Ccdots2);
\path (Cr+1) edge[->] (Cr+2);
\path (Ce+1) edge[->] (Cr+1);
\path (Ce+1) edge[->] (Ce);
\path (Ce) edge[->] (Ce-1);
\path (Ce-1) edge[->] (Ccdots1);
\path (Ccdots1) edge[->] (C2);
\node at (5,0){$r>e$};
\end{tikzpicture}
\end{minipage}
 \begin{minipage}{0.55\textwidth}
 \centering
 \begin{tikzpicture}
\draw (0,0) circle (1.5cm);
  \node[right] at (0:1.7){$0$};
 \node[above] at (45:1.7){$1$};
 \node[below] at (-45:1.7){$e+r$};
 \draw[fill=gray!60,opacity=0.5, draw=none] (0,0) -- (-45:1.5) arc (-45:45:1.5);
 
 \draw[fill=gray!60,opacity=0.5, draw=none] (0,0) -- (140:1.5) arc (140:230:1.5);
 \node at (0:0.8) {$C_0$};
\node at (180:0.8) {$C_\pi$};
 \node[above] at (60:1.7){$2$};
 \node[below] at (250:1.7){$r+1$};
 
  \draw[fill=white] (250:1.5) circle (0.08cm);
 \node[above] at (75:1.7){$3$};
 \node[left] at (140:1.7){$e+1$};
 \node[left] at (200:1.7){$r-1$};
 \node[left] at (230:1.7){$r$};
 \draw[fill=gray!80] (230:1.5) circle (0.08cm);
 \draw[fill=gray!80] (200:1.5) circle (0.08cm);
 \draw[fill=gray!80] (140:1.5) circle (0.08cm);
 \draw[fill=gray!80] (0:1.5) circle (0.08cm);
 \draw[fill=gray!80] (45:1.5) circle (0.08cm);
 \draw[fill=gray!80] (-45:1.5) circle (0.08cm);
  \draw[fill=white] (60:1.5) circle (0.08cm);
   \draw[fill=white] (120:1.5) circle (0.08cm);
   \draw[fill=white] (80:1.5) circle (0.08cm);
 \node[above] at (120:1.7){$e$};
 \draw[dotted,thick] (85:1.7) arc (85:110:1.7);
 \draw[dotted,thick] (150:1.7) arc (150:190:1.7);
 \draw[dotted,thick] (270:1.7) arc (270:295:1.7);
 \end{tikzpicture}
 \end{minipage}
 \caption{Macrostructure of the communicating classes. Transition graph $\calG_C$ of the step-size dynamics of $2$-blocks for $r>e$ (above) with vertex set consisting of the communicating classes which contain either one single state or multiple states as illustrated on the unit circle (below). In particular, $C_{\pi}=\{e+1\}$ for $r=e+1$ and $\#C_\pi>1$ for $r>e+1$.} 
\label{fig:cc}
 \end{figure}
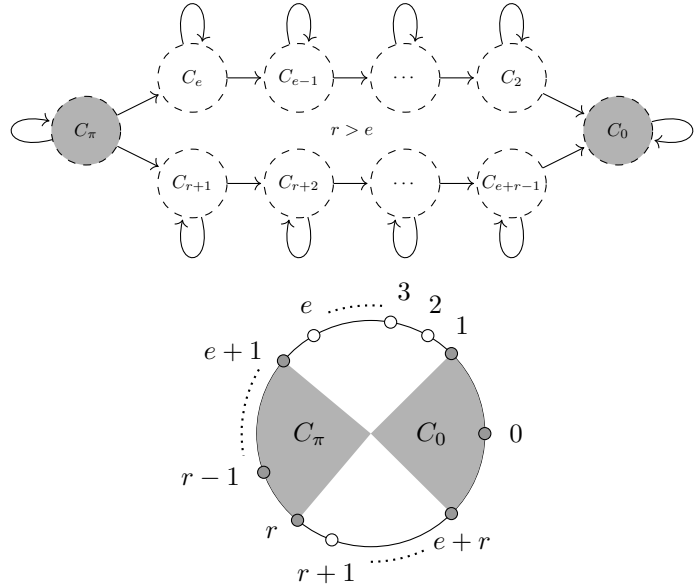

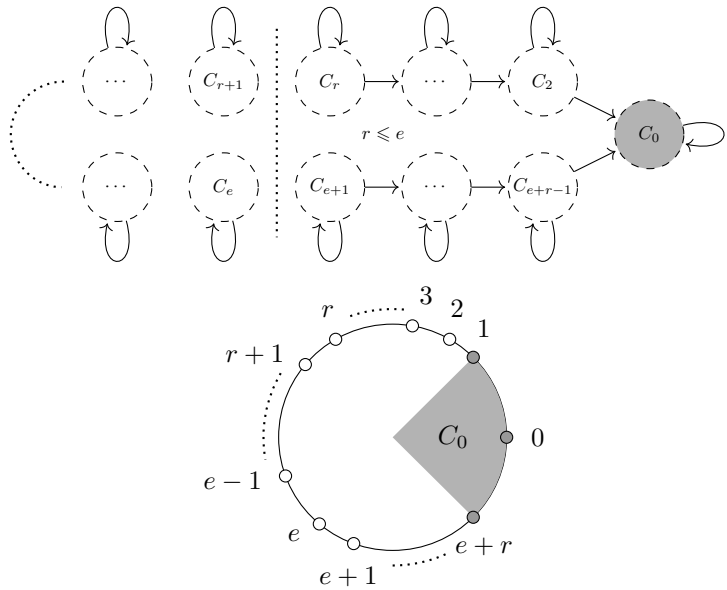
\begin{figure}[H]
\begin{minipage}{0.8\textwidth}
\centering
 \begin{tikzpicture}[scale=0.7, transform shape,shorten >=1pt,]
  \tikzstyle{vertex}=[shape=circle,dashed,draw,minimum size=37pt,inner sep=0pt]

\node[vertex] (Cr+1) at (2,1){$C_{r+1}$};
\node[vertex] (Ce) at (2,-1){$C_e$};
\node[vertex] (Ce+1) at (4,-1){$C_{e+1}$};
\node[vertex] (Cr) at (4,1){$C_{r}$};
\node[vertex, fill=gray!60, opacity=0.5] (C0) at (10,0) {};
\node at (10,0){$C_0$};
\node[vertex] (Ccdots1) at (6,1){$\cdots$};
\node[vertex] (Ccdots2) at (6,-1){$\cdots$};
\node[vertex] (C2) at (8,1){$C_2$};
\node[vertex] (Ce+r-1) at (8,-1){$C_{e+r-1}$};

\node[vertex] (Ccdot3) at (0,1){$\cdots$};
\node[vertex] (Ccdot4) at (0,-1){$\cdots$};

\path (C2) edge[->] (C0);
\path (C2) edge[->, loop above] (C2);
\path (C0) edge[->, loop right] (C0);
\path (Ccdots1) edge[->, loop above] (Ccdots1);
\path (Ccdot3) edge[->, loop above] (Ccdot3);
\path (Ccdot4) edge[->, loop below] (Ccdot4);
\path (Ce) edge[->, loop below] (Ce);
\draw[dotted,thick] (-1,1) arc(90:270:1cm); 
\draw[dotted, thick] (3,2) -- (3,-2); 

\path (Cr+1) edge[->, loop above] (Cr+1);
\path (Cr+2) edge[->, loop below] (Cr+2);
\path (Ccdots2) edge[->, loop below] (Ccdots2);
\path (Ce+r-1) edge[->, loop below] (Ce+r-1);
\path (Cr) edge[->, loop above] (Cr);
\path (Ce+r-1) edge[->] (C0);
\path (Ccdots2) edge[->] (Ce+r-1);
\path (Cr+2) edge[->] (Ccdots2);
\path (Ce-1) edge[->] (Ccdots1);
\path (Ccdots1) edge[->] (C2);
\node at (5,0){$r\leqslant e$};
\end{tikzpicture}
\end{minipage}
\begin{minipage}{0.55\textwidth}
\centering
\begin{tikzpicture}
\draw (0,0) circle (1.5cm);
  \node[right] at (0:1.7){$0$};
 \node[above] at (45:1.7){$1$};
 \node[below] at (-45:1.7){$e+r$};
 \draw[fill=gray!60,opacity=0.5, draw=none] (0,0) -- (-45:1.5) arc (-45:45:1.5);

 \node at (0:0.8) {$C_0$};

 \node[above] at (60:1.7){$2$};
 \node[below] at (250:1.7){$e+1$};
 
  \draw[fill=white] (250:1.5) circle (0.08cm);
 \node[above] at (75:1.7){$3$};
 \node[left] at (140:1.7){$r+1$};
 \node[left] at (200:1.7){$e-1$};
 \node[left] at (230:1.7){$e$};
 \draw[fill=white] (230:1.5) circle (0.08cm);
 \draw[fill=white] (200:1.5) circle (0.08cm);
 \draw[fill=white] (140:1.5) circle (0.08cm);
 \draw[fill=gray!80] (0:1.5) circle (0.08cm);
 \draw[fill=gray!80] (45:1.5) circle (0.08cm);
 \draw[fill=gray!80] (-45:1.5) circle (0.08cm);
  \draw[fill=white] (60:1.5) circle (0.08cm);
   \draw[fill=white] (120:1.5) circle (0.08cm);
   \draw[fill=white] (80:1.5) circle (0.08cm);
 \node[above] at (120:1.7){$r$};
 \draw[dotted,thick] (85:1.7) arc (85:110:1.7);
 \draw[dotted,thick] (150:1.7) arc (150:190:1.7);
 \draw[dotted,thick] (270:1.7) arc (270:295:1.7);
 \end{tikzpicture}
\end{minipage}
\caption{Transition graph $\calG_C$ for $r\leqslant e$. In this case $C_\pi=\emptyset$.}
\label{fig:cc1}
\end{figure}

In what follows, the step-sizes $s_p^m(\cdot)$, as elements of the communicating classes, should not be confused with the states of the cellular automata since both are contained in $\calA$ but obey different transition rules.

\subsection{Local analysis}\label{s:transitions}
In this section, we analyse the dynamics of $2$-blocks $x_{[p,p+1]}^m=(x_p^m,x_{p+1}^m)$ under iterations of $T$. Here, $x_k^m$ is a shorthand for $(T^m(x))_k$ with $k\in\Z$ and $m\in\N_0$. Each such $2$-block has an associated step-size $s_p^m(x)=s(x_p^m,x_{p+1}^m)\in\calA$. For space-time windows $[p,p']\times[m,m']$ of $x$ and $s$ we write, e.g. $x_{[p,p']}^{[m,m']}$.

The following lemma is fundamental to our approach.

\begin{lemma}\label{lem:nontrivial}
The step-size of a $2$-block $x_{[p,p+1]}$ for some $p\in\Z$ changes under $T$ at time $m$, i.e., $s_p^m(x)\neq s_p^{m+1}(x)$, if and only if either $(x_p^m,x_p^{m+1})^\intercal=(0,0)^\intercal$ or $(x_{p+1}^m,x_{p+1}^{m+1})^\intercal=(0,0)^\intercal$. \\
Specifically, the step size changes if and only if either
\[
{(i)} \quad x_{[p,p+1]}^{[m,m+1]} = 
\begin{pmatrix}
0 & c+1\\
0 & c
\end{pmatrix}, \; c\in R\cup\{0\}, 
\quad {(ii)}  \quad x_{[p,p+1]}^{[m,m+1]} =
\begin{pmatrix}
c+1 & 0\\
c & 0
\end{pmatrix}, \; c \in R\cup\{0\}.
\]
Moreover, the step-size change is either:
\begin{itemize}
\item[(i)] An increment and $s_p^{[m,m+1]}(x)=(c,c+1)^\intercal$.
\item[(ii)] A decrement and $s_p^{[m,m+1]}(x)=(-c,-c-1)^\intercal$ with $-c\in\{0,1,\ldots,r\}\Leftrightarrow c\in R\cup\{0\}$.
\end{itemize}
\end{lemma}

\begin{proof}
Let us first suppose that $(x_p^m,x_p^{m+1})^\intercal=(0,0)^\intercal$ and $(x_{p+1}^m,x_{p+1}^{m+1})^\intercal\neq (0,0)^\intercal$. Three further cases have to be distinguished: 
\begin{enumerate}
\item[(1)] $x_{p+1}^m\neq 0$ and $x_{p+1}^{m+1}\neq 0$. Then, $x_{p+1}^m=c\in\calA\setminus\{0,e+r\}$ and, consequently, $x_{p+1}^{m+1}=c+1\in\calA\setminus\{0,1\}$.  
\item[(2)] $x_{p+1}^m=0$ and $x_{p+1}^{m+1}\neq 0$. In this case, we have $x_{p+1}^{m+1}=1$. 
\item[(3)] $x_{p+1}^m\neq 0$ and $x_{p+1}^{m+1}=0$ implies $x_{p+1}^m=e+r$.
\end{enumerate}
In all three cases, the step-size increases and by symmetry, the step-size decreases if $(x_p^m,x_p^{m+1})^\intercal\neq (0,0)^\intercal$ and $(x_{p+1}^m,x_{p+1}^{m+1})^\intercal=(0,0)^\intercal$.

For the other direction, suppose 
\begin{equation}\label{eq:step-size}
s_p^m(x)=x_{p+1}^m-x_p^m\neq x_{p+1}^{m+1}-x_p^{m+1}=s_p^{m+1}(x). 
\end{equation}
For a proof by contradiction, we consider the following two cases:
\begin{enumerate}
\item[(4)] $(x_p^m,x_p^{m+1})^\intercal=(0,0)^\intercal=(x_{p+1}^m,x_{p+1}^{m+1})^\intercal$
\item[(5)] $(x_p^m,x_p^{m+1})^\intercal\neq (0,0)^\intercal\neq (x_{p+1}^m,x_{p+1}^{m+1})^\intercal$
\end{enumerate}
In the first case, condition (\ref{eq:step-size}) does obviously not hold. In the second case, a case-by-case analysis of all possibilities shows that 
\begin{equation*}
((x_p^m,x_p^{m+1})^\intercal,(x_{p+1}^m,x_{p+1}^{m+1})^\intercal)\in\left\{(0,1)^\intercal,(e+r,0)^\intercal,(a,a+1)^\intercal: a\in\calA\setminus\{0,e+r\}\right\}^2.
\end{equation*}
For each such element, it is easy to verify that the step-size remains constant, i.e. $s_p^m(x)=s_p^{m+1}(x)$, contradicting condition (\ref{eq:step-size}).
\end{proof}

The lemma precisely describes how steps size can and cannot change. In order to make the implications transparent, we introduce the concept of communicating classes. For an arbitrary configuration $x\in X$, position $p\in\Z$ and time $m\in\N_0$, we first define the graph of all possible step-size transitions $s_p^m(x)\mapsto s_p^{m+1}(x)$.  

\begin{definition}\label{def:transitiongraph}
We denote by $\calG_s$ the directed graph of local step-size transitions: the set of nodes of $\calG_s$ is  $\calA$, and $\calG_s$ possesses an edge from $s_1$ to $s_2$ for $s_1,s_2\in\calA$ if there is a configuration $x\in X$, a position $p\in\Z$ and a time $m\in \N$ such that $s_p^{[m,m+1]} = (s_1,s_2)^\intercal$. 
In this case we say there exists a transition from $s_1$ to $s_2$ in $\calG_s$ for which we use the notation $s_1\to s_2$; if $s_1=s_2$, a transition is called trivial, otherwise non-trivial. 
\end{definition}
\begin{definition}
Let $s_1,s_2\in\calA$ and $\tau$ be the transition $s_1\to s_2$. $i(\tau)=s_1$ and $t(\tau)=s_2$ are called the \emph{initial} and \emph{terminal state} of $\tau$, respectively. A path $P$ on $\calG_s$ is a finite sequence $P=(\tau_i)_{1\leqslant i\leqslant m}$ of transitions $\tau_i$ such that $t(\tau_i)=i(\tau_{i+1})$ for $1\leqslant i\leqslant m-1$. We say that $s_1$ \emph{communicates} with $s_2$, if there exist paths from $s_1$ to $s_2$ and vice versa, i.e. paths $P=(\tau_i)_{1\leqslant i\leqslant m}$ and $P'=(\tau_i')_{1\leqslant i\leqslant m'}$ with $i(\tau_1)=s_1, t(\tau_m)=s_2$ and $i(\tau_1')=s_2, t(\tau_{m'}')=s_1$. 
\end{definition}

This `communication' is an equivalence relation, hence $\calA$ is partitioned into equivalence classes, referred to as \emph{communicating classes}, with respect to the step-size transitions. The graph $\calG_s$ contains irreducible subgraphs associated with these classes, cf. \cite{lind1995introduction}. In particular, the induced graph $\calG_C$ whose directed edges stem from connections between the communicating classes has \emph{no loops}. Before deriving the edge set and the structure of these equivalence classes, we remark the implications for the non-wandering set.

\begin{remark}\label{rem:wandercomm} Since any 2-block $x_{[p,p+1]}$ of a non-wandering point $x\in\Omega$ must reappear under the dynamics of $T$, so does the step-size $s_1=s_p(x)$. Hence, if the step size changes to a value $s_2$, which does not communicate with $s_1$, then $x\not\in \Omega$. In other words, if $x\in\Omega$ and $p\in\Z$, the step-sizes $s_p^m(x)$ for any $m\in\N$ lie in the same communicating class.\\
\end{remark}

\begin{lemma}\label{l:comm}
Let $s_1, s_2\in \calA$. Then $\calG_s$ has an edge from $s_1$ to $s_2$ if and only if either
\begin{itemize}
\item[(i)] $s_2=s_1+1$ and $s_1\in\{0,e+1,\ldots,e+r\}=R\cup\{0\}$,
\item[(ii)] $s_2=s_1-1$ and $s_1\in\{0,1,\ldots,r\}$,
\item[(iii)] $s_2=s_1$.
\end{itemize}
Specifically, $s_1\to s_2$ if and only if either $s_1=s_2$ or $|s_2-s_1|=1$ and $s_1,s_2\in C_0\cup C_{\pi}$, where $C_0\coloneqq \{e+r,0,1\}$ and $C_\pi\coloneqq \{e+1,\ldots, r\}$ in case $r>e$, while $C_{\pi}=\emptyset$ otherwise.
In particular, the communicating classes of $\calG_s$ are 
\[
C_0,\; C_\pi,\; C_a\coloneqq \{a\}, a\in\calA\setminus(C_0\cup C_\pi).
\]
\end{lemma}
\begin{proof}
The cases (i), (ii) are a direct consequence of Lemma~\ref{lem:nontrivial}. Case (iii) follows from $s_p^m(x) = s_p^{m+1}(x)$ if $x_p^m=x_{p+1}^m>0$ and the possibility that $x_p^{m+1}=x_{p+1}^{m+1}=0$ in case $x_p^m=x_{p+1}^m=0$.

Case (iii) means trivial transitions always occur, and it follows that a non-trivial transition is present precisely when cases (i) and (ii) occur jointly (i.e. the intervals overlap, cf. Figure~\ref{f:overlapping}). This in turn implies the communicating classes, cf. Figure~\ref{fig:cc}. 
\end{proof}

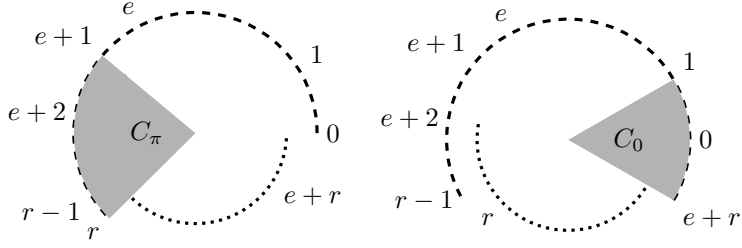
\begin{figure}[h]
\begin{minipage}{0.35\textwidth}
\begin{center}
\begin{tikzpicture}[scale=0.8]

  \draw[dashed, very thick] (2,0) arc (0:225:2cm); 
    
   \draw[dotted, very thick] (140:1.5) arc (140:360:1.5cm);
   \node[right] at (0:2) {$0$};
    \node[above right] at (30:2) {$1$};
     \node[above left] at (115:2) {$e$};
      \node[above left] at (140:2) {$e+1$};
      \node[below left] at (225:2) {$r$};
       \node[below right] at (330:1.5) {$e+r$};
        \node[below left] at (210:2) {$r-1$};
        \node[left] at (170:2) {$e+2$};
       \draw[fill=gray!60,draw=none,opacity=0.5] (0,0) -- (140:2) arc (140:225:2);
       
        \node at (-0.8,0) {$C_\pi$};
       \end{tikzpicture}
\end{center}
\end{minipage}
\begin{minipage}{0.45\textwidth}
\begin{center}
\begin{tikzpicture}[scale=0.8]

  \draw[dashed, very thick] (-30:2) arc (-30:210:2cm); 
   \draw[dotted, very thick] (-190:1.5) arc (-190:30:1.5cm);
   \node[right] at (0:2) {$0$};
   
    \node[above right] at (30:2) {$1$};
     \node[above left] at (115:2) {$e$};
      \node[above left] at (140:2) {$e+1$};
      \node[below left] at (225:1.5) {$r$};
       \node[below right] at (330:2) {$e+r$};
       \node[left] at (210:2) {$r-1$};
       \node[left] at (170:2) {$e+2$};
       \draw[fill=gray!60,draw=none,opacity=0.5] (0,0) -- (-30:2) arc (-30:30:2);
        \node at (1,0) {$C_0$};
       \end{tikzpicture}
\end{center}
\end{minipage}
\caption{The inner dotted arcs represent step-sizes $s_1\in R\cup\{0\}$ (left) and their increments $s_2=s_1+1$ (right). The outer dashed arcs show the step-sizes $s_1\in\{0,1,\ldots,r\}$ (left) and their decrements $s_2=s_1-1$ (right). The maximal overlapping is given by $C_0\cup C_\pi$ (shadowed region).}
\label{f:overlapping}
\end{figure}

\begin{remark}
It follows that for $x\in\Omega$ with $s_p(x)\in C_a$, $a\notin\{0,\pi\}$, the step size at $p$ remains constant. Step size transitions are possible within $C_0$ and, if $r> e+1$, within $C_\pi$ only. Note that for $e+1=r$ we have $C_\pi=\{e+1\}$ so that the step-size cannot change in this class.
\end{remark}

\begin{remark}\label{rem:independence}
If $s_{p,m}(x)\in C_\pi$ then for $c$ in Lemma~\ref{lem:nontrivial} we have $c\in C_{e+r}\setminus\{r\}$ since otherwise $s_p^{m+1}(x)=r+1\in C_{r+1}$; analogously $c\neq e+1$. Therefore, anytime $x_p^m=0$ (resp.\ $x_{p+1}^m=0$), the neighbor state $x_{p+1}^m$ (resp.\ $x_p^m)$ is a refractory state, which implies that the dynamics on the space-time windows $[-\infty,p]\times \N$ and $[p+1,+\infty]\times\N$ are independent of each other. 
\end{remark}

Let us describe the non-trivial transitions in more detail; for generality, we consider configurations $x\in\calA^\mathbb{K}, \mathbb{K}\in\{\Z,\Z_{\leqslant q},\Z_{\geqslant q}:q\in\Z\}$.

\begin{definition}\label{transitionsets}
For $x\in\calA^\mathbb{K}$, $k\in\N$ and $p,p+1\in\mathbb{K}$, the transition times at $p$ in $C_j$ is the finite or infinite sequence $D_{p,j}(x)=(m_1,m_2,\ldots)$, $m_i\in\N$ with $m_i<m_{i+1}$ such that 
\begin{enumerate}
\item[(a)] $s_p^{m_i}(x)\in C_j$ for all $i\in\{1,2,\ldots,k\}$,
\item[(b)] $s_p^{m}(x)\neq s_p^{m-1}(x)$ $\Leftrightarrow$ $m=m_i$ for some $i\in\N$. 
\end{enumerate}
The transition time $m_i$ called consecutive if $m_{i+1}=m_{i}+1$ or $m_{i-1}=m_{i}+1$ and separated otherwise.
Times $m>0$ such that $s_p^{m}(x)=s_p^0(x)$ are called step returns.
\end{definition}

\begin{remark}\label{r:transitiontimes}
Let $r_k$ be a step return time of some $x\in\calA^\mathbb{K}$.
Since step size changes are by $\pm1$, for the transition times $m_i\leqslant r$ we have
\begin{equation}\label{eq:Tr}
\#\{m_i\leqslant r_k: s_p^{m_i}(x)=s_p^{m_{i-1}}(x)+1\}=\#\{m_i\leqslant r_k: s_p^{m_i}(x)=s_p^{m_{i-1}}(x)-1\},
\end{equation}
Due to Remark~\ref{rem:wandercomm} and Remark~\ref{rem:cylsets}, for $x\in \Omega$ the sequence of step-returns is infinite.
\end{remark}

\begin{cor}\label{cor:trtup}
For $x\in\calA^\mathbb{K}$, each separated transition time $m_i$ lies in a triple 
\begin{align*}
&(x_p^{m_i-1},x_p^{m_i},x_p^{m_i+1})^\intercal\in\{(0,0,0)^\intercal, (0,0,1)^\intercal\}\quad\textnormal{or}\\ &(x_{p+1}^{m_i-1},x_{p+1}^{m_i},x_{p+1}^{m_i+1})^\intercal\in\{(0,0,0)^\intercal,(0,0,1)^\intercal\},
\end{align*}
depending on whether $m_i$ is associated with a step-size in $C_0$ or, if $r>e+1$, in $C_{\pi}$.
If $r>e+1$, there can be at most $\max\{2,\# C_\pi-1\}$ consecutive transition times $m_i<m_{i+1}<\ldots<m_j$. More specifically, if the step-sizes associated to the $m_k, i\leqslant k\leqslant j$, are in $C_{\pi}$, the $m_k$ lie in a block
\begin{equation}\label{transtimes}
(x_p^{m_i-1},\ldots,x_p^{m_j+1})^\intercal=(0_{\ell},1)^\intercal\quad\textnormal{or}\quad (x_{p+1}^{m_i-1},\ldots,x_{p+1}^{m_j+1})^\intercal=(0_{\ell},1)^\intercal
\end{equation}
with a zero block $0_{\ell}$ of length $3\leqslant\ell\leqslant \# C_\pi$ while there can be exactly two consecutive transition times $m_1,m_2$ with associated step-sizes in $C_0$ which lie in a quadruple of type (\ref{transtimes}) or a triple
\begin{equation*}
 (x_p^{m_1-1},x_p^{m_1},x_p^{m_2})^\intercal=(0,0,1)^\intercal\textrm{ or }(x_{p+1}^{m_1-1},x_{p+1}^{m_1},x_{p+1}^{m_2})^\intercal=(0,0,1)^\intercal.
\end{equation*}
\end{cor}
\begin{proof}
This is a direct consequence of Lemma~\ref{l:comm} noting that consecutive transitions increase the length of the zero block in time so that $c$ is incremented further, but these increments cannot be in $E$ before a transition.
\end{proof}

We end this subsection with examples for the macro- and microstructure of the communicating classes in case $r>e+1$ and $e>r$. In the next section, we use this local framework of the step-size dynamics to determine the global spatial structure of non-wandering points.

\begin{example}\label{example1}
Let $e=3$ and $r=7$. The communicating classes are given by $C_0=\{0,1,10\}, C_2=\{2\}, C_3=\{3\}, C_\pi=\{4,5,6,7\}, C_8=\{8\}$ and $C_9=\{9\}$, see Figure~\ref{transgraph}. 
\end{example}
 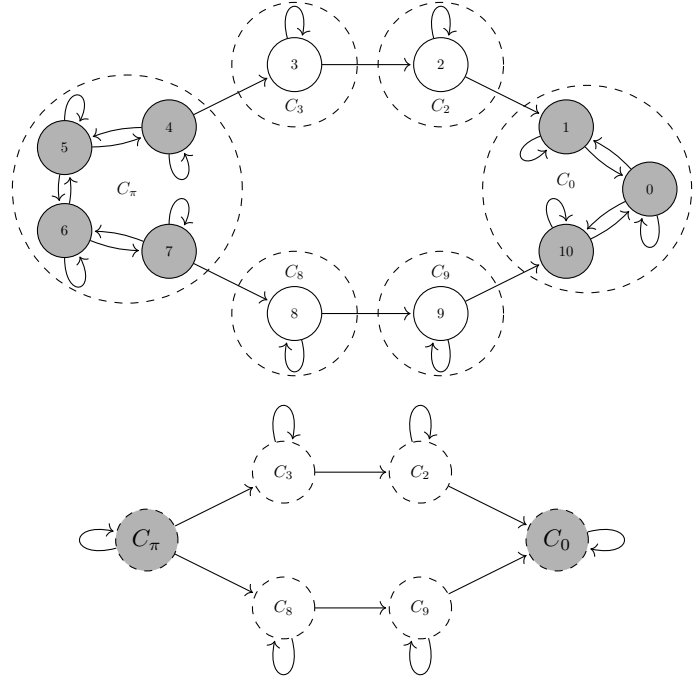
\begin{figure}[t]
 \begin{center}
 \begin{tabular}{c}
 \begin{tikzpicture}[scale=0.55,transform shape,shorten >=1pt,->]
  \tikzstyle{vertex}=[circle,,draw,minimum size=37pt,inner sep=0pt]
  \node[vertex,fill=gray!60,opacity=0.5] (0) at (10,0) {};
  \node at (10,0){$0$};
  \node[vertex,fill=gray!60,opacity=0.5] (1) at (8,1.5) {};
  \node at (8,1.5){$1$};
  \node[vertex,fill=gray!60,opacity=0.5] (10) at (8,-1.5) {};
  \node at (8,-1.5){$10$};
  \node at (8,0.2) {\Large{$C_0$}};
  \draw[dashed,] (8.5,0) circle (2.5cm);
  \path (0) edge[->, bend right=10] (1);
  \path (1) edge[->, bend right=10] (0);
  
  \path (0) edge[->, bend right=10] (10);
  \path (10) edge[->, bend right=10] (0);
  
  \path (0) edge[->, loop below] (0);
  \draw (1) edge[->, loop, out=200, in=230,looseness=7] (1);
  \draw (10) edge[->, loop, out=120, in=90, looseness=7] (10);
\node[vertex] (2) at (5,3) {$2$};
  \node at (5,2) {\Large{$C_2$}};
  \draw[dashed] (5,3) circle (1.5cm);
  \node[vertex] (3) at (1.5,3) {$3$};
  \node at (1.5,2) {\Large{$C_3$}};
  \draw[dashed] (1.5,3) circle (1.5cm);
  \path (2) edge[->] (1); 
  \path (3) edge[->] (2);
  
  \path (2) edge[->, loop above] (2);
  \path (3) edge[->, loop above] (3);
  
  \node[vertex,fill=gray!60,opacity=0.5] (5) at (-4,1) {};
  \node at (-4,1){$5$};
  \node[vertex,fill=gray!60,opacity=0.5] (4) at (-1.5,1.5) {};
  \node at (-1.5,1.5){$4$};
  \node[vertex,fill=gray!60,opacity=0.5] (6) at (-4,-1) {};
  \node at(-4,-1){$6$};
  \node[vertex,fill=gray!60,opacity=0.5] (7) at (-1.5,-1.5) {};
  \node at (-1.5,-1.5){$7$};
  \draw[dashed] (-2.5,0) circle (2.75cm);
  \path (4) edge[->] (3);
  \draw (4) edge[->, loop, out=270, in=300,looseness=7] (4);
  \draw (5) edge[->, loop, out=90,in=60,looseness=7] (5);
  \draw (6) edge[->, loop, out=270,in=300,looseness=8] (6);
  \draw (7) edge[->, loop,out=90,in=60,looseness=7] (7);
  \node at (-2.5,0) {\Large{$C_{\pi}$}};
  \path (4) edge[->,bend right=10] (5);
  \path (5) edge[->,bend right=10] (4);
  
  \path (5) edge[->, bend right=10] (6);
  \path (6) edge[->, bend right=10] (5);
  
  \path (6) edge[->, bend right=10] (7);
  \path (7) edge[->, bend right=10] (6);
  
  \node[vertex] (8) at (1.5,-3) {$8$};
  \node at (1.5,-2) {\Large{$C_8$}};
  \draw[dashed] (1.5,-3) circle (1.5cm);
  \path (7) edge[->] (8);
  
  \node[vertex] (9) at (5,-3) {$9$};
  \node at (5,-2) {\Large{$C_9$}};
  \draw[dashed] (5,-3) circle (1.5cm);
  \path (8) edge[->] (9);
  \path (9) edge[->] (10);
  \path (9) edge[->, loop below] (9);
  \path (8) edge[->, loop below] (8);
  \end{tikzpicture}\\
  \begin{tikzpicture}[scale=0.9,transform shape,shorten >=1pt,->]

  \tikzstyle{vertex}=[scale=.7,shape=circle,dashed,draw,minimum size=37pt,inner sep=0pt]
\node[vertex,fill=gray!60,opacity=0.5] (Ce+1) at (-4,0){};
\node at (-4,0){$C_{\pi}$};
\node[vertex] (Ce) at (-2,1){$C_3$};
\node[vertex] (Cr+1) at (-2,-1){$C_8$};
\node[vertex,fill=gray!60,opacity=0.5] (C1) at (2,0) {};
\node at (2,0){$C_0$};
\node[vertex] (C2) at (0,1){$C_2$};
\node[vertex] (Ce+r-1) at (0,-1){$C_9$};
\path (Cr+1) edge[->] (Ce+r-1);
\path (Ce+1) edge[->] (Ce);
\path (Ce+1) edge[->] (Cr+1);
\path (Ce+r-1) edge[->] (C1);
\path (C2) edge[->] (C1);
\path (Ce+1) edge[->, loop left] (Ce+1);
\path (C1) edge[->, loop right] (C1);
\path (C2) edge[->, loop above] (C2);
\path (Ce) edge[->] (C2);
\path (Ce) edge[->, loop above] (Ce);
\path (Cr+1) edge[->, loop below] (Cr+1);
\path (Ce+r-1) edge[->, loop below] (Ce+r-1);
\end{tikzpicture}
\end{tabular}
\end{center}
\caption{Example~\ref{example1}. The transition graph $\mathcal{G}_s$ with vertex set $\calA$ (above) reveals the microstructure of the communicating classes of the step-size transitions while the graph $\mathcal{G}_C$ (below) illustrates the macrostructure by using the set of communicating classes as the vertex set, cf. Figure~\ref{fig:cc}; for visibility the dashed circles representing $C_0$ and $C\pi$ are not filled in gray in the upper picture.} 
\label{transgraph}
\end{figure}
\begin{example}\label{example2}
Let $e=7$ and $r=4$. The communicating classes are given by $C_0=\{0,1,11\}, C_2=\{2\}, C_3=\{3\},C_4=\{4\}, C_5=\{5\}, C_6=\{6\}, C_7=\{7\}, C_8=\{8\}, C_9=\{9\}$ and $C_{10}=\{10\}$, see Figure~\ref{fig:rsmallerase}. 
\end{example}

\begin{figure}
\begin{center}
\begin{tabular}{c}
\begin{tikzpicture}[scale=0.52, transform shape,shorten >=1pt,->]
  \tikzstyle{vertex}=[circle,draw,minimum size=37pt,inner sep=0pt]
  \node[vertex,fill=gray!60,opacity=0.5] (0) at (10,0) {};
  \node at(10,0){$0$};
  \node[vertex,fill=gray!60,opacity=0.5] (1) at (8,1.5) {};
  \node at (8,1.5){$1$};
  \node[vertex,fill=gray!60,opacity=0.5] (11) at (8,-1.5) {};
  \node at (8,-1.5){$11$};
  \node at (8,0.2) {\Large{$C_0$}};
  \draw[dashed] (8.5,0) circle (2.5cm);
  \path (0) edge[->, bend right=10] (1);
  \path (1) edge[->, bend right=10] (0);
  
  \path (0) edge[->, bend right=10] (11);
  \path (11) edge[->, bend right=10] (0);
  
  \path (0) edge[->, loop below] (0);
  \draw (1) edge[->, loop, out=200, in=230,looseness=7] (1);
  \draw (11) edge[->, loop, out=120, in=90, looseness=7] (11);
  
  \node[vertex] (2) at (5,3) {$2$};
  \node at (5,2) {\Large{$C_2$}};
  \draw[dashed] (5,3) circle (1.5cm);
  \node[vertex] (3) at (1.5,3) {$3$};
  \node at (1.5,2) {\Large{$C_3$}};
  \draw[dashed] (1.5,3) circle (1.5cm);
  \path (2) edge[->] (1); 
  \path (3) edge[->] (2);
  \node[vertex] (4) at (-2,3) {$4$};
  \node at (-2,2) {\Large{$C_4$}};
  \draw[dashed] (-2,3) circle (1.5cm);
   \node[vertex] (10) at (5,-3) {$10$};
  \node at (5,-2) {\Large{$C_{10}$}};
  \draw[dashed] (5,-3) circle (1.5cm);
   \node[vertex] (9) at (1.5,-3) {$9$};
  \node at (1.5,-2) {\Large{$C_9$}};
  \draw[dashed] (1.5,-3) circle (1.5cm);
   \path (10) edge[->] (11);
   \path (9) edge[->] (10);
   \path (4) edge[->] (3);
    \node[vertex] (8) at (-2,-3) {$8$};
   \node at (-2,-2) {\Large{$C_{8}$}};
   \draw[dashed] (-2,-3) circle (1.5cm);
   \path (8) edge[->] (9);
  \path (8) edge[->, loop below] (8);
  \path (9) edge[->, loop below] (9);
  \path (10) edge[->, loop below] (10);
  \path (2) edge[->, loop above] (2);
  \path (3) edge[->, loop above] (3);
  \path (4) edge[->, loop above] (4);
  \node[vertex] (7) at (-5.5,-3) {$7$};
  \draw[dashed] (-5.5,-3) circle (1.5cm);
  \node at (-5.5,-2) {\Large{$C_{7}$}};
  \node[vertex] (5) at (-5.5,3) {$5$};
  \draw[dashed] (-5.5,3) circle (1.5cm);
  \node at (-5.5,2) {\Large{$C_{5}$}};
  \path (7) edge[->, loop below] (7);
  \path (5) edge[->, loop above] (5);
  \node[vertex] (6) at (-9,3) {$6$};
  \draw[dashed] (-9,3) circle (1.5cm);
  \node at (-9,2) {\Large{$C_{6}$}};
  \path (6) edge[->, loop above] (6);
\end{tikzpicture}\\
\begin{tikzpicture}[scale=0.7, transform shape,shorten >=1pt]
  \tikzstyle{vertex}=[shape=circle,dashed,draw,minimum size=37pt,inner sep=0pt]
\node[vertex,fill=gray!60,opacity=0.5] (C0) at (2,0) {};
\node at (2,0){$C_0$};
\node[vertex] (C2) at (0,1){$C_2$};
\node[vertex] (C3) at (-2,1){$C_3$};
\node[vertex] (C4) at (-4,1){$C_4$};
\node[vertex] (C5) at (-6,1){$C_5$};
\node[vertex] (C10) at (0,-1){$C_{10}$};
\node[vertex] (C9) at (-2,-1){$C_{9}$};
\node[vertex] (C8) at (-4,-1){$C_{8}$};
\node[vertex] (C7) at (-6,-1){$C_{7}$};
\node[vertex] (C6) at (-8,1){$C_{6}$};
\draw[dotted, thick] (-5,2) -- (-5,-2);
\path (C0) edge[->, loop right] (C0);
\path (C2) edge[->, loop above] (C2);
\path (C3) edge[->, loop above] (C3);
\path (C4) edge[->, loop above] (C4);
\path (C5) edge[->, loop above] (C5);
\path (C6) edge[->,loop left] (C6);
\path (C7) edge[->,loop below] (C7);
\path (C8) edge[->, loop below] (C8);
\path (C9) edge[->, loop below] (C9);
\path (C10) edge[->, loop below] (C10);
\path (C2) edge[->] (C0);
\path (C3) edge[->] (C2);
\path (C4) edge[->] (C3);
\path (C8) edge[->] (C9);
\path (C9) edge[->] (C10);
\path (C10) edge[->] (C0);
\end{tikzpicture}
\end{tabular}
\end{center}
\caption{Example~\ref{example2}.  $\mathcal{G}_s$ (above) and $\calG_C$ (below), cf. Figure~\ref{fig:cc1}}
\label{fig:rsmallerase}
\end{figure}
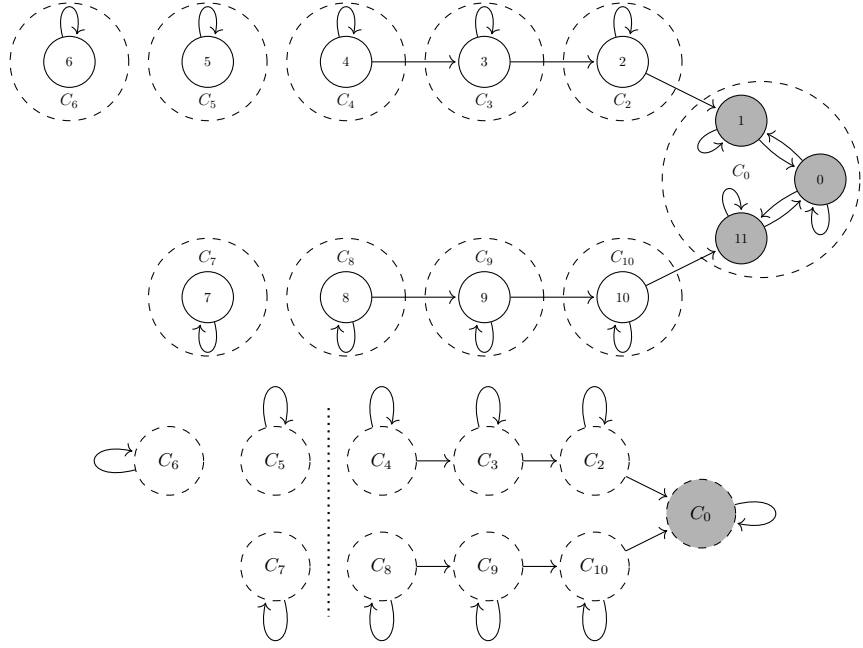
\subsection{Global analysis: characterising the non-wandering set}\label{sec:2options}

The local analysis showed that for $x\in\Omega$ and any $p\in\Z$, the step-sizes $s_p^m(x)$ for all $m\in \N$ lie in the same communicating class. We also know that all $k$-blocks $x_{[p,p+k]}$, $k\in\N$ need to reappear infinitely often under $T$. In this section we infer from this the spatial structure of $x\in \Omega$. Specifically, we show that the class $C_0$ can be identified with the pulse collision subsystem $Z$ of \S\ref{sec:collisions}, and $C_\pi$ with the non-trivial dislocations of \S\ref{s:disloc}. To this end, we use the following result which characterizes the non-wandering points with step-sizes in $C_0$ only.  
\begin{lemma}\label{l:ZC0}
$x\in\Omega$ and $s_p(x)\in C_0$ for all $p\in\Z$ if and only if $x\in Z$.
\end{lemma}
\begin{proof}
By definition of $Z$ it follows that $s_p(x)\in C_0$ for all $p\in\Z$ and we showed $Z\subset \Omega$ in Proposition~\ref{prop:ZinNW}. For the converse, suppose $x\in\Omega$ and $s_p(x)\in C_0$ for all $p\in\Z$. The only blocks which might occur in $x$ but not in elements of $Z$ are 3-blocks contained in
\begin{align*}
    F_{3,1}:=&\{(a,a,a): a\in E\cup R\},\\
    F_{3,2}:=&\{(a,a,b)\in (E\cup R)^2\times\calA: s(a,b)=e+r\}\\
    &\cup\{(a,b,b)\in\calA\times (E\cup R)^2: s(b,a)=e+r\},\\
    F_{3,3}:=&\{(a,b,c)\in\calA^3: s(b,a)=e+r\textrm{ and }s(b,c)=e+r\}
\end{align*}
and $n$-blocks $(e+r,0_{n-2},e+r)$, $n\geqslant 4$.  

It follows from the analysis of forbidden blocks in \cite{KRU} that these blocks do not occur in non-wandering points, but for completeness we give the proof in the present more special case in Appendix A.
\end{proof}

We next identify trivial, $\CardA$-periodic dislocations, but first make a simple, but fundamental observation.

\begin{lemma}\label{l:C0er}
If $x\in\Omega$ and $s_p^m(x)$ is not constant in $m\in\N$ for some $p$, then $s_q^0(x)\in C_0\cup C_\pi$ for all $q\in\Z$ and there is $m\in\N$ such that $x_p^{[m,m+1]} = (0,0)^\intercal$.
\end{lemma}
\begin{proof}
The block $x^m_{[p-1,p+1]}$ repeats infinitely many times for $m\in\N$ and any step-size change of $s_p^m(x)$ is compensated by a reverse step, cf.\ Remark~\ref{r:transitiontimes}. Hence, there is $m\in \N$ such that $x_p^{[m,m+1]} = (0,0)^\intercal$. If $x_{p-1}^m= 0$ we immediately have $s_{p-1}^0(x)\in C_0$. If $x_{p-1}^m\neq 0$ then $s_{p-1}^m(x)\neq s_{p-1}^{m+1}$ and by Lemma~\ref{l:comm} it follows that $s_{p-1}^0\in C_0\cup C_\pi$, which are the only classes that allow for changing step sizes. The claim follows by induction on the position.
\end{proof}

The following lemma shows how `excitations' can be backtracked in time by a spatial shift in one direction, if all step size lie in $C_0\cup C_\pi$. By the previous lemma this is the case if the step-size varies somewhere.

\begin{lemma}\label{l:excite}
Suppose $x\in\Omega$, $s_p^0(x)\in C_0\cup C_\pi$ for all $p\in\Z$. 
If there is $m\in\N, p\in\Z$ such that $x_{p,p+1}^{m}=(1,0)$ then for all $1\leqslant j\leqslant m$ it holds that $x_{p-j,p+1-j}^{m-j}=(1,0)$. Likewise, $x_{p,p+1}^{m}=(0,1)$ implies $x_{p+j,p+1+j}^{m'-j}=(0,1)$ for all $1\leqslant j\leqslant m$.
\end{lemma}

\begin{proof}
By assumption $x_{[p-1,p+1]}^{[m-1,m]} = \begin{pmatrix} a_3&1&0\\a_2&0&a_1\end{pmatrix}$  for some $a_1,a_2,a_3\in\calA$. By the step-size assumption we have $a_1\in C_0\cup C_\pi$ and since $x_{p+1}^m=0$ we have $a_1\in\{0,e+r\}$. Therefore, $x_p^m=1$ requires $a_2\in E$ and by the step-size assumption $a_2=1$. Hence, $x_{p-1,p}^{m-1} = (1,0)$. The claim follows by induction, and by spatial reflection symmetry.
\end{proof}

\begin{lemma}\label{l:zerocolumn}
If $x\in\Omega$ and there exist $p\in\Z, m\in\N_0$ such that $x_p^k=0$ for all $k\geqslant m$, then $x^m\in Z$.
\end{lemma}
\begin{proof}
Without loss of generality, suppose $x_p^k=0$ for all $k\geqslant 0$. By induction over the position,  
\begin{equation}\label{zerocond}
\forall q\in\Z~\exists n=n(q)\geqslant 0: x_{q}^k=0\textrm{ for all } k\geqslant n
\end{equation}
since otherwise $x_p^k\neq 0$ for some $k\geqslant 0$. 

Suppose $x$ is not the zero configuration (which would already mean $x\in Z$), and let $q_{\min}\in\Z$ be the nearest position to $p$ such that $x_q\neq 0$. Without loss of generality, assume $q_{\min}>p$. Then $x_q^k=0$ for all $p\leqslant q<q_{\min}$ and all $k\geqslant 0$. Consequently, $s_{q_{\min}-1}(x)=a\in R$ and there exists a change in the step-size, $s_{q_{\min}-1}^1(x)=a+1\neq a$. Hence, $s_q(x)\in C_0\cup C_{\pi}$ for all $q\in\Z$ by Lemma~\ref{l:C0er}. However, by (\ref{zerocond}), for each step-size $s_{q}(x)\in C_{\pi}$ there exists some $k\geqslant n(q)$ such that $s_q^k(x)\notin C_{\pi}$, contradicting $x\in\Omega$. Hence, $s_q(x)\in C_0$ for all $q\in\Z$ which, together with $x\in\Omega$, implies $x\in Z$ by Lemma~\ref{l:ZC0}. 
\end{proof}

\subsubsection{Constant step size}\label{s:constantstepsize}

\begin{lemma}\label{lem:conststepsize}
Let $x\in\Omega\setminus Z$ with $s_p^m(x)=s_p^0(x)~\forall m\in\N$ and some fixed $p\in\Z$. Then
\begin{equation*}
x\in \OmConst\coloneqq \{x\in \Omega\setminus Z~|~\forall q\in\Z, m\in\N: s_q^m(x)=s_q^0(x)\}.
\end{equation*}
Moreover, $x\in X\setminus Z$ is $\CardA$-periodic if and only if $x\in\OmConst$.
\end{lemma}

\begin{proof}
Suppose $x\in\Omega$ and there exists some $p\in\Z$ such that $s_p^m(x)=c\in C_j$ for all $m\in\N_0$. By Lemma~\ref{lem:nontrivial} for any $m\in\N$ there are two possibilities for constant step-size dynamics:
\begin{enumerate}
\item[(1)] $x_p^{[m,m+1]}\neq (0,0)^\intercal\neq x_{p+1}^{[m,m+1]}$ for all $m\in\N_0$. 
\item[(2)] $x_p^{[m,m+1]}= (0,0)^\intercal= x_{p+1}^{[m,m+1]}$ for some $m\in\N_0$.  
\end{enumerate}
Case (1). By Lemma~\ref{l:C0er} applied to $p\pm1$, any step-size change of $s_{p-1}(x)$ or $s_{p+1}(x)$ implies a block $x_p^{[m,m+1]}=(0,0)^\intercal$ for some $m\in \N$, which does not occur in case (1). By induction over $q\in\Z$, $s_q^{m}(x)$ is constant for all $m\in\N_0$ and $q\in\Z$. In particular, $x$ is $\CardA$-periodic due to the absence of any zero block $0_{\ell}, \ell\geqslant 2$, in $(x_q^m)_{m\in\N}$ for any $q$. 

Case (2). We show that in this case $x\in Z$ so that $x\notin \OmConst$. The constant step-size is $c=0\in C_0$, i.e., $x_p^k=x_{p+1}^k$ for all $k\in\N_0$. 

First, assume that $x_{[p,p+1]}^{m'}=(0,0)$ for all $m'\geqslant m$. By Lemma~\ref{l:zerocolumn}, $x^m\in Z$, i.e. $s_p(x^m)\in C_0$ for all $p\in\Z$ and hence $s_p(x)\in C_0$ for all $p\in\Z$, meaning that $x\in Z$ (Lemma~\ref{l:ZC0}). 

Next, suppose there exists some (smallest) $m'>m+1$ with $x_{[p,p+1]}^{m'}=(1,1)$, that is, there is a window
\begin{equation}
    x_{[p-1,p+1]}^{[m'-2,m']}=\begin{pmatrix}a_3 & 1 & 1\\a_2 & 0 & 0\\a_1 & 0 & 0\end{pmatrix}
\end{equation}
for some $a_1,a_2,a_3\in\calA$. Since $x_p^{m'}=1$ and $x_{p+1}^{m'-1}=0$, we must have $a_2\in E$ and $x_p^{m'-1}=0$ implies $a_2=1$ so that $a_1=0$. This means $s_{p-1}^{m'-2}(x)\neq s_{p-1}^{m'-1}(x)$ so that by Lemma~\ref{l:C0er} we know $s_q^k(x)\in C_0\cup C_\pi$ for all $q\in\Z$ and $k\in\N_0$. By Lemma~\ref{l:excite}, the block $x_{[p-1,p]}^{m'-1}=(1,0)$ shifts back spatially to $x_{[p-m'-2,p-m'-1]}=(1,0)$. Hence, we can conclude that $s_q(x)\in C_0$ for all $p-m'-2\leqslant q\leqslant p$. 

To conclude that $s_q(x)\in C_0$ for all $q\in\Z$, note that we can assume that there are infinitely many $k>m'$ with $x_{[p-1,p+1]}^{[k,k+1]}=\begin{pmatrix}a_1 & 1 & 1\\a_2 &0 & 0\end{pmatrix}$; otherwise we are again in the situation of Lemma~\ref{l:zerocolumn}. Since we have already concluded that all step-sizes lie in $C_0\cup C_{\pi}$, we must have $a_2=1$ and $a_1=2$ and, again by Lemma~\ref{l:C0er}, the block $x_{[p-1,p]}^k=(1,0)$ shifts back to $x_{[p-1-k,p-k]}=(1,0)$. Since this holds for infinitely many increasing $k>m'$ one deduces $s_q(x)\in C_0$ for all $q\leqslant p$. By symmetry, also $s_q(x)\in C_0$ for $q>p$, i.e. $x\in Z$ by Lemma~\ref{l:ZC0}.
\end{proof}

\begin{remark}\label{r:converseconclusion}
By reverse conclusion, any non-wandering point in $\Omega\setminus Z$ with some varying step size has nowhere constant step size. Lemma~\ref{lem:conststepsize} shows in particular that $\OmConst$ is a $T$-invariant subset of $\Omega$. 
\end{remark}

\subsubsection{Varying step size}\label{s:varyingstepsize}
Next, we let $x\in\Omega\setminus Z$ and suppose that for some position $p\in\Z$ the step-size varies (i.e. $\exists m\in\N: s_p^m(x)\neq s_p^0(x)=c$). By the previous section, this implies $x\in\OmConst^C$ and it turns out that $x\in\OmConst^C\setminus Z$ is the following set.
\begin{align*}
\OmVar&\coloneqq \bigcup_{p\in\Z}\{x\in \Omega\setminus\OmConst: s_p^0(x)\in C_{\pi},x_{[-\infty,p]}\in S_{\textrm{R},p,-}^{\# C_\pi},x_{[p+1,+\infty]}\in S_{\textrm{L},p+1,+}^{\# C_{\pi}}\},\\
S_{\textrm{L},p,+}^{k}&\coloneqq \{x\in\SLppos: \textnormal{ at most }k\textnormal{ consecutive zeros occur in }x\},\\
S_{\textrm{R},p,-}^{k}&\coloneqq \{x\in\SRpneg: \textnormal{ at most }k\textnormal{ consecutive zeros occur in }x\}.
\end{align*}
for $k\in\N_0$.

\begin{lemma}\label{l:var}
If $x\in\Omega\setminus Z$ and $s_p^m(x)$ is not constant in $m\in\N$ for some $p\in\Z$ and $m\in\N$ then $x\in \OmVar$ and $r>e+1$.  
\end{lemma}

\begin{proof}
Recall that the only two options for varying step sizes are $c\in C_0$ and, if $r>e+1$, $c\in C_\pi$ as for $r=e+1$ we have $C_\pi=\{e+1\}$, which does not allow for step-size variation. Hence, $\OmVar=\emptyset$ in case $r\leqslant e+1$.

Since $x$ is not identically zero, there must be a window 
\[
x_{[q,q+1]}^{[k,k+1]} \in\left\{\begin{pmatrix} 1&a+1\\0&a\end{pmatrix}, \begin{pmatrix} a+1&1\\a&0\end{pmatrix}\right\}.
\]
If $a\in E\neq 1$ we have $s_q^k(x)\in\calA\setminus (C_0\cup C_\pi)$, but by Lemma~\ref{l:C0er} it holds that $s_q^0(x)\in C_0\cup C_\pi$ for all $q\in \Z$. Therefore $a=1$ and as in the proof of Lemma~\ref{lem:conststepsize} by backtracking excitation loops in time it follows that $x_{[-\infty,p]}\in\SRpneg$ and $x_{[p+1,+\infty]}\in S_{\textrm{L},p+1}^{+}$. If $s_q^0(x)\in C_0$ for all $q\in\Z$ then $x\in Z$, i.e., $x\notin\OmVar$. Therefore, there is $q\in\Z$ with $s_q^0(x)\in C_{\pi}$.

In this case, we can backtrack analogously to Lemma~\ref{lem:conststepsize}; however, the restriction $2\leqslant\ell\leqslant\# C_{\pi}$ on zero blocks $0_\ell$ occuring in $(x_q^m(0))_{m\in\N}$, cf.\ Corollary~\ref{cor:trtup}, implies that at most $\# C_{\pi}$ consecutive zeros can occur in $x$ , i.e.  $x_{[p+1,+\infty]}\in S_{\textrm{L},p+1,+}^{\# C_\pi}$ and $x_{[-\infty,p]}\in S_{\textrm{R},p,-}^{\# C_{\pi}}$, hence $x\in\OmVar$.
\end{proof}

Using that $C_{\pi}=\emptyset$ for $e+1>r$,  which means $\OmVar=\emptyset$, the lemmas together prove in particular Theorem~\ref{NWeleqr}.

\subsection{Computing the topological entropy: proof of Theorem~\ref{maintheorem}}
In this section, we finally determine the topological entropy $h(X,T)$ of $T$ and use that $h(X,T)=h(\Omega,T|_{\Omega})$. 

By Theorem~\ref{NWeleqr} the non-wandering set is the union of disjoint $T$-invariant sets, $\Omega=Z\uplus\OmConst\uplus\OmVar$, where $\OmVar=\emptyset$ for $e\geqslant r+1$ and $\OmVar\neq\emptyset$ otherwise, so that 
\begin{equation}\label{maxentropy}
h(\Omega,T|_{\Omega})=\max\{h(Z,T|_{Z}),h(\OmConst,T|_{\OmConst}), h(\OmVar,T|_{\OmVar})\}.
\end{equation}

By Proposition~\ref{p:entropy}, $h(Z,T|_{Z})=2\ln\rho_{e+r}$ and it remains to consider $h(\OmConst,T|_{\OmConst})$ and $h(\OmVar,T|_{\OmVar})$. 

As to $\OmConst$, by Lemma~\ref{lem:conststepsize} this set contains $\CardA$-periodic configurations only, which directly implies 
\begin{equation}\label{eq:TEofOmConst}
h(\OmConst,T|_{\OmConst})=0. 
\end{equation}
As to $\OmVar$, we assume $r>e+1$ in the following and note $\#C_{\pi}=r-e>1$ in this case. By Remark~\ref{rem:independence}, $T|_{\OmVar}$ splits into two independent dynamics on the lattices to the left and right of the particular unique stationary dislocation $p\in\Z$ with step size $s_p^0(x)\in C_{\pi}$. As the position of the dislocation is stationary, we can make use of the skew-product structure established in \textsection\ref{sec:skew} in order to find an upper estimate for $h(\OmVar,T|_{\OmVar})$.  To this end, let
\begin{align*}
Q_0\coloneqq &\left\{(x,y)\in S_{\textrm{R},0,-}^{\# C_\pi}\times S_{\textrm{L},1,+}^{\# C_{\pi}}: y_{1}=x_0+c, c\in C_{\pi}\right\}\subset\Sigma_-\times\Sigma_+,\\
W\coloneqq &Q_0\times\Z,\\
V\coloneqq &\OmVar\times\Z.
\end{align*}
On $V$ and $W$, respectively, consider the maps 
\begin{align*}
&\f\coloneqq T|_{\OmVar}\times\textnormal{ id}\colon V\to V,\\
&g\coloneqq (\sigRneg\times\sigLpos)\times\textnormal{ id}\colon W\to W. 
\end{align*}
For $x\in\OmVar$, let $p(x)\in\Z$ denote its unique stationary separating position (dislocation) with step-size $s_{p(x)}^0(x)\in C_{\pi}$. We define $\varphi\colon V\to W$ by
\begin{equation*}
\varphi((x,p(x)))\coloneqq (x_{[-\infty,p(x)]},x_{[p(x)+1,+\infty]},p(x)).
\end{equation*}
$(V,f)$ and $(W,g)$ are topologically conjugate, i.e. the diagram
\begin{equation*}\label{diagram3}
\begin{CD}
V     @>\f>>  V\\
@V\varphi VV @VV\varphi V\\
W     @>g>>  W
\end{CD}   
\end{equation*}
commutes, $\varphi\circ \f=g\circ\varphi$, where $\varphi$ is a homeomorphism. Thus, $h(V,\f)=h(W,g)$. Since the identity map has zero topological entropy,
\begin{equation*}
h(\OmVar,T|_{\OmVar})=h(V,\f)=h(W,g)=h(Q_0,\sigRneg\times\sigLpos).
\end{equation*}
Note $Q_0\subset\Sigma_-\times\Sigma_+$ is an invariant (proper) subset with respect to the product map $\sigRneg\times\sigLpos$. Using the product rule of topological entropy, 
\begin{align*}
h(Q_0,\sigRneg\times\sigLpos)&\leqslant h(\Sigma_-\times\Sigma_+,\sigRneg\times\sigLpos)=h(\Sigma_-,\sigRneg)+h(\Sigma_+,\sigLpos)=2\ln\rho_{e+r}.
\end{align*}
In fact, we next show that the upper bound is strictly smaller than $2\ln\rho_{e+r}$. In preparation, note $Q_0$ is a $(\sigRneg\times\sigLpos)$-invariant subset of $S_{\textrm{R},0,-}^{\# C_{\pi}}\times S_{\textrm{L},1,+}^{\# C_{\pi}}$, i.e.
\begin{equation*}
h(Q_0,\sigRneg\times\sigLpos)\leqslant h(S_{\textrm{R},0,-}^{\# C_{\pi}}\times S_{\textrm{L},1,+}^{\# C_{\pi}},\sigRneg\times\sigLpos)=h(S_{\textrm{R},0,-}^{\# C_{\pi}},\sigRneg)+h(S_{\textrm{L},1,+}^{\# C_{\pi}},\sigLpos).
\end{equation*}
The topological entropies on the RHS of the equation can be computed by considering semi-infinite walks on the graph corresponding to the alphabet
\begin{equation*}
\calA'\coloneqq (\calA\setminus\{0\})\cup\{0^i: 1\leqslant i\leqslant\#C_{\pi}\}
\end{equation*}
and the dynamics resulting from replacing the trivial transition $0\to 0$ in Figure~\ref{f:trans} by transitions  
\begin{align*}
0^i&\to 1\text{ for all }1\leqslant i\leqslant\#C_{\pi},\\
0^i&\to 0^j\text{ exactly if }1\leqslant i,j\leqslant\#C_{\pi}\text{ and }j=i+1,
\end{align*}
cf. Figure~\ref{f:transnew}. 
\begin{figure}%[h]
\begin{center}
\begin{tikzpicture}[scale=0.6, transform shape,shorten >=1pt,->]
  \tikzstyle{vertex}=[circle,fill=black!15,minimum size=37pt,inner sep=0pt]
 \node[vertex] (3) at (0,0) {$\cdots$};
  \node[vertex] (2) at (2,1)    {$2$};
  \node[vertex] (1) at (4,1) {$1$};
  \node[vertex] (0) at (6,0) {$0^1$};
  \node[vertex] (e+r) at (6,-2) {$e+r$};
  \node[vertex] (6) at (4,-3) {$\cdots$};
  \node[vertex] (e+1) at (2,-3) {$e+1$};
  \node[vertex] (e) at (0,-2) {$e$};
  \path (0) edge[->] (1);
  \path (1) edge[->] (2);
  \path (2) edge[->] (3);
  \path (3) edge[->] (e);
  \path (e) edge[->] (e+1);
  \path (e+1) edge[->] (6);
  \path (6) edge[->] (e+r);
  \path (e+r) edge[->] (0);
 \node[vertex] (02) at (8,0) {$0^2$};
 \node[vertex] (03) at (10,0) {$0^3$};
  \node[vertex] (04) at (12,0) {$\ldots$};
  \node[vertex] (0r-e-1) at (14,0) {$0^{r-e-1}$};
  \node[vertex] (0r-e) at (16,0) {$0^{r-e}$};
  
 %% Edges 
  \path (0) edge[->] (02);
 \path (02) edge[->,bend right] (1.0);
  \path (02) edge[->] (03);
  \path (03) edge[->,bend right] (1.20);
  \path (03) edge[->] (04);
  \path (04) edge[->,bend right] (1.40);
  \path (0r-e-1) edge[->,bend right] (1.60);
  \path (0r-e) edge[->,bend right] (1.80);
  \path (0r-e-1) edge[->] (0r-e);
  \path (04) edge[->] (0r-e-1);
\end{tikzpicture}
\end{center}
\caption{Transition graph for the one-sided subshifts $S_{\textrm{R},p,-}^{\# C_{\pi}}$ and $S_{\textrm{L},p,+}^{\# C_{\pi}}$.}
\label{f:transnew}
\end{figure}
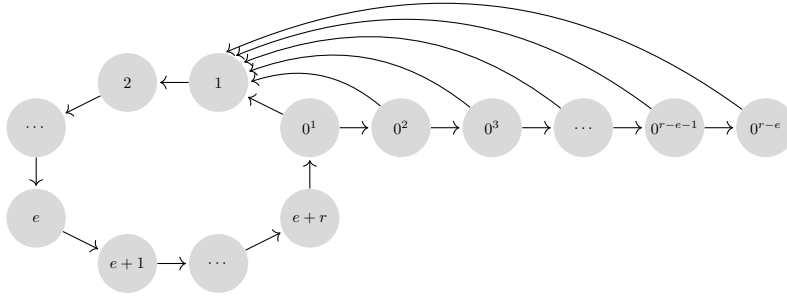
The resulting transition matrix $M\in\{0,1\}^{2r\times 2r}$ is of block form
\begin{equation*}
M=\begin{pmatrix}A & B\\C &D\end{pmatrix}
\end{equation*}
with quadratic matrices $A\in\{0,1\}^{(\#C_{\pi})^2}$ and $D\in\{0,1\}^{(e+r)^2}$ with $1$'s on the upper off-diagonal and $0$'s elsewhere. The matrix $B\in\{0,1\}^{\#C_{\pi}\times e+r}$ is composed of the first column $(1,1,\ldots,1)^T$ and $0$'s elsewhere. Finally, the matrix $C\in\{0,1\}^{e+r\times\#C_{\pi}}$ has a $1$ in the last entry of the first column (which corresponds to the edge $e+r\to 0^1$) and is $0$ otherwise.

In order to determine the characteristic polynomial of $M$, let $I$ denote the unity matrix of suitable dimension and  note $D-\lambda I$ is invertible for $\lambda\neq 0$; its inverse matrix is given by
\begin{equation}\label{Dinverse}
    (D-\lambda I)^{-1}=-\begin{pmatrix}\lambda^{-1} & \lambda^{-2}& \lambda^{-3} & \hdots & \hdots & \lambda^{-(e+r)}\\ & \lambda^{-1} & \lambda^{-2} & \lambda^{-3} & \hdots & \lambda^{-(e+r-1)}\\& & \ddots & \ddots & \ddots & \vdots\\  &  &  &\ddots & \ddots & \lambda^{-3}\\ &  &  &  &  & \lambda^{-2}\\ &  &  &  &  & \lambda^{-1} \end{pmatrix},
\end{equation}
where blank entries are supposed to mean $0$. The product of $(D-\lambda I)^{-1}$ and $C$ is an $(e+r)\times\#C_{\pi}$-matrix whose first column is given by the last column of (\ref{Dinverse}) and is $0$ otherwise. Multiplying by matrix $B$ gives the $\#C_{\pi}\times\#C_{\pi}$-matrix $Q\coloneqq B(D-\lambda I)^{-1}C$ whose first column is given by $-(\lambda^{-(e+r)},\ldots,\lambda^{-(e+r)})^T$ and is $0$ otherwise. Subtracting $Q$ from $A-\lambda I$ yields 
\begin{equation*}
    P\coloneqq (A-\lambda I)-Q=\begin{pmatrix}-\lambda+\lambda^{-(e+r)} & 1 &  &  &  & \\\lambda^{-(e+r)} & -\lambda & 1 &  & & \\\lambda^{-(e+r)} &  & -\lambda & 1 &  & \\\vdots &  &  & \ddots & \ddots & \\\lambda^{-(e+r)} & & & & -\lambda& 1\\\lambda^{-(e+r)} & & & & &-\lambda\end{pmatrix}
\end{equation*}
whose determinant can be computed by expanding along the first row. By $\#C_{\pi}=r-e$,
\begin{equation}\label{detP}
    \lvert P\rvert=\lambda^{-(e+r)}(\lambda^{2r}-\lambda^{r-e-1})-\begin{vmatrix}\lambda^{-(e+r)} & 1 & 0 &\hdots & 0\\\vdots & -\lambda & 1 & \hdots & 0\\
    \vdots & 0 & -\lambda & \ddots &\\\lambda^{-(e+r)} & & & \ddots&  1\\\lambda^{-(e+r)} & & & & -\lambda\end{vmatrix}
\end{equation}
The determinant on the RHS of (\ref{detP}) can be computed again by expanding along the first row, yielding the summand $\lambda^{-(e+r)}(-\lambda)^{r-e-2}$ and another determinant computable by expanding along the first row. Repeating this until we end up with the summand $\lambda^{-(e+r)}$ the determinant of $P$ is eventually given by
\begin{equation*}
    \lvert P\rvert= \lambda^{-(e+r)}\cdot\left(\lambda^{2r}-\sum_{i=0}^{\#C_{\pi}-1}\lambda^{i}\right)
\end{equation*}
Using $\lvert D-\lambda I\rvert=(-\lambda)^{e+r}$ and the determinant formula for block matrices,  
\begin{align*}
    \lvert M-\lambda I\rvert&=\lvert D-\lambda I\rvert\cdot\rvert P\rvert=(-\lambda)^{e+r}\lambda^{-(e+r)}\left(\lambda^{2r}-\sum_{i=0}^{\#C_{\pi}-1}\lambda^i\right)\\
    &=(-1)^{e+r}\left(\lambda^{2r}-\sum_{i=0}^{r-e-1}\lambda^i\right).
\end{align*}
Since we are interested in the roots of this polynomial, we can neglect the factor $(-1)^{e+r}$ and consider the polynomial
\begin{equation*}
\fer(\lambda)=\lambda^{2r}-\sum_{i=0}^{r-e-1}\lambda^i
\end{equation*}
which has exactly one positive (simple) root $\etaer$ by Descartes' rule and $\etaer>1$ since $\fer(0), \fer(1)<0$ and the leading coefficient is positive. Hence,
\begin{equation}\label{eq:topEntEta}
h(S_{\textrm{R},0,-}^{\# C_{\pi}},\sigRneg)=h(S_{\textrm{L},1,+}^{\# C_{\pi}},\sigLpos)=\ln\etaer.
\end{equation}

Finally, we show $\etaer<\rho_{e+r}$ and hence $h(\OmVar,T|_{\OmVar})<2\ln(\rho_{e+r})$.  First note $\fer(\lambda) = \lambda^{2r} -\frac{\lambda^{r-e}-1}{\lambda-1}=\lambda^{r-e}f_c(\lambda)+1$ where $f_c$ is the polynomial from Lemma~\ref{l:asy} with unique positive root $\rho_c$. Consequently, $\fer(\rho_c)=1>0$. This implies $\etaer<\rho_{e+r}$ since the root $\etaer>1$ of $\fer$ is unique and its leading coefficient is positive, so that $\fer(\lambda)>0,\lambda>0\Rightarrow \lambda>\etaer$. In particular, $1<\etaer<\rho_c$, i.e. $\etaer\to 1$ as $e\to\infty$ by Lemma~\ref{l:asy}.

This concludes the proof of Theorem~\ref{maintheorem}.

\bigskip
Recall that the topological entropy $h(Z,T_{|Z})$ depends on $e,r$ only through $e+r$, cf. Lemma~\ref{l:asy}. We end this section by commenting on scaling bounds of \eqref{eq:topEntEta}. 

\begin{remark}\label{rem:asy} The asymptotic scaling of (\ref{eq:topEntEta}) and, consequently, of the upper bound $2\ln\etaer$ of $h(\OmVar,T_{|\OmVar})$, differs from the scaling of $2\ln(\rho_{r+e})=h(Z,T_{|Z})$, and in particular depends on the difference $r-e$. 

More specifically, we next show that if $r-e$ is constant then $\zeta:=\etaer-1\sim\frac 1 r$ as $r\to\infty$, which is thus asymptotic smaller than $\rho_{r+e}\sim \frac{\ln(r+e)}{r+e}$; recall $0<\zeta<\frac{\ln (e+r)}{e+r}$ so $\zeta\to0$ as $r+e\to\infty$. In particular, a time-rescaled topological entropy on $\OmVar$ as in Corollary~\ref{Cor:scaling} behaves differently than that on $Z$.

As to the proof, we rewrite $\fer(1+\zeta)=0$ and take logarithms as follows:
\begin{align*}
(1+\zeta)^{r-e} = 1+ \zeta(1+\zeta)^{2r}
&\Leftrightarrow (r-e)\ln(1+\zeta) = \ln(1+\zeta(1+\zeta)^{2r})\\
&\Leftrightarrow (r-e) = \frac{\ln(1+\zeta(1+\zeta)^{2r})}{\ln(1+\zeta)},
\end{align*}
Hence, if $r-e$ is constant then, as $r\to\infty$, we have $\zeta(1+\zeta)^{2r}\sim\zeta$, i.e., $(1+\zeta)^{2r}\sim 1$ and therefore $\zeta\sim \frac 1 r$.

In contrast, if $e$ is constant one can show that $\eta\sim \frac{\ln r}{r}$, i.e. the upper bound $\ln(\eta)$ for the topological entropy on $\OmVar$ scales as that on $Z$. However, we omit this as we do not investigate a lower bound here.
% Idea: substitute \zeta = ln r/r and show LHS and RHS of 1st equation above scale the same as r\to \infty (namely as r ln r)
\end{remark}
%%%%% OUTLOOK %%%%%
\section{Outlook and discussion}

In this paper we have investigated the recurrent dynamics of the general Green\-berg-Hastings cellular automata for excitable media. It turned out that the non-wandering set decomposes into invariant subsets that each generate different topological entropy and each can be identified with different wave dynamics. The largest complexity in this sense rests in the pulse collision subsystem, which is also chaotic in the sense of Devaney.  The counter-propagating pulse dynamics can be captured by a conjugacy to a skew-product system of purely left- and right-moving pulses. The other non-trivial complexity can be interpreted as wave dynamics of defect and dislocation-type, and has a Markovian structure. The relation to wave dynamics and more specifically the decomposition of the non-wandering set came as a surprise to us, and we believe this in itself makes it an interesting example as a dynamical system. Indeed, CA can sometimes be used as fruitful formulations of dynamical systems \cite{muller_spandl_2009} (note the erratum \cite{muller_spandl_2010}).

\medskip
Several open questions and avenues for further investigations remain, and we just briefly mention some. We initially hoped that a skew-product formulation can be used to directly compute the topological entropy, but it is complicated by the translations after pulse annihilation that have no a priori bound. In fact, we are not aware of results for ergodic properties of this kind of skew-product dynamics. Regarding complexity, a natural line of research concerns metric entropies and other complexity measures such as Lyapunov exponents, and refined ergodic properties such as a thermodynamic formalism. Again the skew-product structure gives a heuristic guideline, but it seems difficult to exploit. We note that for the Greenberg-Hastings automaton with $e=r=1$ the Bernoulli measure has been considered in \cite{DS91} and some attractor properties of a variant of this CA haven been studied in \cite{hurley_1990} (example 4C, p.\ 681). %here the set \Lambda(f)$ is the eventual image, p. 676
We also refer \cite{BLANCHARD199786} as a (not very recent) review on one-dimensional CA, and the references therein.

\medskip
As mentioned in the introduction, the wave dynamics in some partial differential equation models of excitable media numerically shows pulses split and do not annihilate upon collision. This can generate intricate patterns of annihilation and replication in the space-time plane, e.g., Sierpinsky-gaskets. In order to incorporate aspects of this one can modify the rules of the automaton; preliminary results are promising. While it is well known that CA with two states can generate such patterns, the relation to models of excitable media is unclear to our knowledge. Indeed, the `strong interaction' of localised patterns in PDE is notoriously difficult and the CA models do not directly help with the technical difficulties of the continuum problems. Perhaps it may be possible to eventually quantitatively relate relatively simple CA models to such a continuum limit.

We believe our results show that aspects of the complexity of strong interaction can be analysed and related to wave dynamcis. However, some of our proofs involve tedious combinatorics and seem to dependent on the specific rules and it  would be interesting to understand whether there is some robustness in the methods.

%JR: Continuum limit is a big issue and I am not sure we should mention this here as we would need to cite a lot and it really goes far beyond the paper... More generally, it would be interesting to understand better quantitative relations of continuum PDE models and discrete models of excitable media. Several paper have studied  purely spatial discretisation

%%%%%%%% Appendix %%%%%%%
\appendix

\section{Proof of Lemma~\ref{l:ZC0}}
In order to prove Lemma~\ref{l:ZC0} on p.~\pageref{l:ZC0} we use the following easy but helpful observation. In the following $x^{-j}$ denotes an arbitrary element of $T^{-j}(\{x\})$, if it exists, and $x^{-j}_n$ its $n$-th coordinate.
\begin{lemma}\label{l:prep}
For   $a-1\in E \cup R$ and $k\in \Z$  consider $x\in X$  with either $x_{[k,k+1]} = (0,a)$ or $x_{[k-1,k]} = (a,0)$.
If $x^{1-a}_{k} = c$ exists, then
\begin{enumerate}
\item[(i)] $c=e+r-a+2$ if $a-1\in E $.
\item[(ii)] $c\in[e+r-a+2,r+1]$ if $a-1\in R$.
\end{enumerate}
\end{lemma}

\begin{proof}
In case (i) we have $x^{-1}_k=e+r$ by the local preimage formula (\ref{pre-image}), p.\ \pageref{pre-image}, and further local preimages at $k$ are unique for $j=-2,\ldots,e+r-1$ and simply decremental: $x^{-j}_k=\CardA-j$ so that at $j=a-1$ we have $x^{-j}_k=e+r-a$.

In case (ii) the preimage $x^{-1}_k$ might be $e+r$ or $0$. In the first case the same as in (i) applies, which gives the lower bound $c=e+r-a+2$. Increasing $j$, it might be that $x^{-j}_k=0$ until $j=a-e-1$ for which $x^{-j}_{k\pm1} = e+1$ (with sign depending on the 2-block). In that case we have a block of type (i) with $a=e+1$ for which $c=e+r-a+2=r+1$, which is the upper bound (note $a>e+1$ in the present case).
\end{proof}

\begin{proof}[Proof  of Lemma~\ref{l:ZC0} (forbidden blocks).]
We complete the proof of Lemma~\ref{l:ZC0} by first showing that 3-blocks in
\begin {enumerate}
\item[(i)] $\calF_{3,1}:=\{(a,a,a): a\in E\cup R\}$
\item[(ii)]$\calF_{3,2}:=\left\{(a,a,b)\in (E\cup R)^2\times\mathcal{A}\colon s(a,b)>e\right\}$

\qquad$\cup\left\{(a,b,b)\in\mathcal{A}\times (E\cup R)^2\colon s(b,a)>e\right\}$,

\item[(iii)] $\calF_{3,3}:=\left\{(a,b,c)\in\mathcal{A}^3\colon  s(b,a)>e\text{ and }s(b,c)>e\right\}$
\end{enumerate}
do not occur in elements of the eventual image $Y$ (and hence not in elements of the non-wandering set $\Omega$, cf. Lemma~\ref{lem:NWinY}); in particular, the blocks mentioned in the proof of Lemma~\ref{l:ZC0} satisfy $F_{3,k}\subset\calF_{3,k}, k=1,2,3$. 

The remaining $n$-blocks $(e+r,0_{n-2},e+r)$, $n\geqslant 4$ will be treated at the end.

\medskip
As a direct consequence of the local preimage formula \eqref{pre-image} on page \pageref{pre-image}, 3-blocks in
\begin{equation*}
 \calF_{3,0}:=(\calA\setminus (E+1))\times\{1\}\times (\calA\setminus (E+1))
\end{equation*}
do not occur in elements of the eventual image. 
\vspace{0.3cm}

\textbf{ad (i).} Suppose $x\in X$ with $x_{[k,k+2]}=(a,a,a)$ for some $k\in\mathbb{Z}$ and $a\in E\cup R $. Then $x^{-(a-1)}_{[k,k+2]}=(1,1,1)$ which lies in $\calF_{3,0}$ so that $(a,a,a)$ cannot occur in the eventual image.
\vspace{0.3cm}

\textbf{ad (ii).} By symmetry, it suffices to prove the statement for triples $(a,a,b)$ with $a\in E\cup R$ and $b\in\mathcal{A}$ such that $s(a,b)>e$. Suppose $x\in X$ with $x_{[k,k+2]}=(a,a,b)$ for such a triple and some $k\in\mathbb{Z}$. We claim that for any choice of preimages we have $x^{1-a}_{[k,k+2]}=(1,1,c)$ with $c\in\mathcal{A}\setminus (E+1)$, which lies in $\calF_{3,0}$ so that $(a,a,b)$ cannot occur in the eventual image. 

For $a<b$ the first $a-1$ preimages are unique so $x_{[k,k+2]}^{1-a}=(1,1,b-a+1)$. Since $b-a+1>e+1$ by assumption, this block lies in $\calF_{3,0}$. 

In case $a>b$ the first $b-1$ preimages are unique  and assuming that the $b$-th preimage exists (i.e. the right neighbour lies in $E+1$) gives $x_{[k,k+2]}^{-b}=(a-b,a-b,0)$ with $a-b\in[2,r]$ (which requires $r>1$). Lemma~\ref{l:prep} implies $x_{[k,k+2]}^{1-a}=(1,1,c)$ with $c\in[e+r-(a-b),r+2]$ so that $c\in [e+2,r+1]$. But then $x_{[k,k+2]}^{1-a}$ lies in $\calF_{3,0}$.
\vspace{0.3cm}

\textbf{ad (iii).} Let $(a,b,c)\in \calF_{3,3}$ and $x\in X$ with $x_{[k,k+2]}=(a,b,c)$ for some $k\in\mathbb{Z}$. 
\vspace{0.1cm}

\textbf{(iii.1)} In case $b\in R$, the assumption $s(b,a), s(b,c)>e$ means $a,c\in[b-r,b-1]$, and, without loss of generality by spatial reflection, $a\leqslant c$. Then $x^{-a}_{[k,k+2]}=(0,b-a,c-a)$ with $b-a\in[1,r]$ and $e<s(b,c)=c-b<c-a$ so $c-a\notin E+1$. Hence, if $b-a=1$ then $(0,b-a,c-a)\in F_{3,0}$ so we may assume $b-a\in[2,r]$. It then follows from Lemma~\ref{l:prep} that $x^{1-b}_{[k,k+2]} = (a',1,c')$ with $a',c'\geqslant e+r+2-(b-a)\geqslant e+2$ and thus $(a',1,c')\in \calF_{3,0}$.
\vspace{0.1cm}

\textbf{(iii.2)} In case $b\in E$, the assumption $s(b,a), s(b,c)>e$ means $a,c\in[0,b-1]\cup[b+e+1,e+r]$. 
The subcase $a,c\in[0,b-1]$, wlog $a\leqslant c$, yields $x^{-a}_{[k,k+2]}=(0,b-a,c-a)$ with $b-a\in[1,e]\subset [1,r]$ so that case 1 can be applied. If $c\in[b+e+1,e+r]$ we obtain $c'=c-b+1>e+1$ with $c'$ derived in case 1, and thus also $(a',1,c')\in \calF_{3,0}$. 
The last subcase $a,c\in[b+e+1,e+r]$ means $x^{1-b}_{[k,k+2]} = (a-b+1,1,c-b+1)\in \calF_{3,0}$ as in the previous subcase.
\vspace{0.1cm}

\textbf{(iii.3)} Finally, assume $b=0$ which implies $a,c\in R$, wlog $a\leqslant c$. By Lemma~\ref{l:prep} we have $x^{1-a}_{[k,k+2]} = (1,b',c-a+1)=:(a',b',c')$ with $b' \in [e+r+2-a,r+1]\subset[2,r+1]$ and $c' \in[1,\CardA-a]\subset[1,r+1]$. Thus $a'-b'\in [-r,a-\CardA]\subset[-r,-1]$ and $c'-b' \in [-r,-1]$, and therefore $s(b',a')=\CardA+a'-b' >e$, and $s(b',c')=\CardA+c'-b'>e$.

\bigskip
Next, we show by induction that $n$-blocks, $n\geqslant 4$, of the form $(a,0,\ldots,0,b)$ in 
\begin{equation*}
\calF_n:=R\times\{0\}^{n-2}\times R
\end{equation*}
do not occur in elements of the eventual image.

Consider $n=4$ first and let $x\in X$ with $x_{[k,k+3]}=(a,0,0,b)\in\calF_4$ for some $k\in\Z$, where wlog $a\leqslant b$. In case $a=b=e+1$ we have $x_{[k,k+2]}^{-1}=(e,e+r,e+r)\in\calF_{3,2}$. If $a=e+1$ and $r>e+1$ we either have $x_{[k+1,k+3]}^{-1}=(e+r,0,b-1)\in\calF_{3,3}$ or $x_{[k+1,k+3]}^{-1}=(e+r,e+r,b-1)\in\calF_{3,2}$. If $b\geqslant a>e+1$ we have 
\begin{equation*}
x_{[k,k+3]}^{-1}\in\{(a-1,c,d,b-1): c,d\in\{0,e+r\}\}
\end{equation*}
and for $c\neq 0$ or $d\neq 0$ we get 3-blocks in $\calF_{3,2}\cup\calF_{3,3}$ while for $c=d=0$ the 4-block is contained in $\calF_4$. Hence, as long as $x_k^{-j}>e+1$, the preimage contains a block with non-empty preimage. Eventually, $x_k^{-j}=e+1$ and we are in one of the previous cases.

Suppose now the statement holds for some $n\geqslant 4$ and let $x\in X$ with $x_{[k,k+n]}=(a,0_{n-1},b)\in\calF_{n+1}$ for some $k\in\Z$ and wlog $a\leqslant b$. If $a=e+1$, we either have $x_{[k,k+2]}^{-1}\in\calF_{3,2}$, $x_{[k+1,k+3]}^{-1}\in\calF_{3,3}$ or $x_{[k+1,k+n]}^{-1}\in\calF_n$. If $a>e+1$  we either have one of the previous blocks or $x_{[k,k+n]}^{-1}=(a-1,0_{n-1},b-1)\in\calF_{n+1}$ and this repeats until we end up with $x_k^{-j}=a-j=e+1$ which is the previous case.
\end{proof}


\begin{thebibliography}{10}

\bibitem{Adler}
R.~L. Adler, A.~G. Konheim, and M.~H. McAndrew.
\newblock Topological entropy.
\newblock {\em Transactions of the American Mathematical Society},
  114(2):309--319, 1965.

\bibitem{BLANCHARD199786}
F.~Blanchard, P.~Kurka, and A.~Maass.
\newblock Topological and measure-theoretic properties of one-dimensional
  cellular automata.
\newblock {\em Physica D: Nonlinear Phenomena}, 103(1):86--99, 1997.
\newblock Lattice Dynamics.

\bibitem{bowen1971entropy}
R.~Bowen.
\newblock Entropy for group endomorphisms and homogeneous spaces.
\newblock {\em Transactions of the American Mathematical Society},
  153:401--414, 1971.

\bibitem{Bufetov1999}
A.~Bufetov.
\newblock Topological entropy of free semigroup actions and skew-product
  transformations.
\newblock {\em Journal of Dynamical and Control Systems}, 5(1):137--143, Jan
  1999.

\bibitem{canovas2013topological}
J.~S. C{\'a}novas.
\newblock On the topological entropy of some skew-product maps.
\newblock {\em Entropy}, 15(8):3100--3108, 2013.

\bibitem{dinaburg1970correlation}
E.~I. Dinaburg.
\newblock A correlation between topological entropy and metric entropy.
\newblock In {\em Doklady Akademii Nauk}, volume 190, pages 19--22. Russian
  Academy of Sciences, 1970.

\bibitem{DS91}
R.~Durrett and J.~E. Steif.
\newblock Some rigorous results for the {G}reenberg-{H}astings model.
\newblock {\em J. Theoret. Probab.}, 4(4):669--690, 1991.

\bibitem{GHH78}
J.~M. Greenberg, B.~D. Hassard, and S.~P. Hastings.
\newblock Pattern formation and periodic structures in systems modeled by
  reaction-diffusion equations.
\newblock {\em Bull. Amer. Math. Soc.}, 84(6):1296--1327, 1978.

\bibitem{hasselblatt2005topological}
B.~Hasselblatt, Z.~Nitecki, and J.~Propp.
\newblock Topological entropy for non-uniformly continuous maps.
\newblock {\em arXiv preprint math/0511495}, 2005.

\bibitem{hurd_kari_culik_1992}
L.~P. Hurd, J.~Kari, and K.~Culik.
\newblock The topological entropy of cellular automata is uncomputable.
\newblock {\em Ergodic Theory and Dynamical Systems}, 12(2):255--265, 1992.

\bibitem{hurley_1990}
M.~Hurley.
\newblock Ergodic aspects of cellular automata.
\newblock {\em Ergodic Theory and Dynamical Systems}, 10(4):671--685, 1990.

\bibitem{Izhikevich}
E.~M. Izhikevich.
\newblock {\em Dynamical Systems in Neuroscience: The geometry of excitability
  and bursting}.
\newblock MIT Press, 2010.

\bibitem{katok1997introduction}
A.~Katok and B.~Hasselblatt.
\newblock {\em Introduction to the Modern Theory of Dynamical Systems}.
\newblock Encyclopedia of Mathematics and its Applications. Cambridge
  University Press, 1997.

\bibitem{KRU}
M.~Kesseb\"ohmer, J.~Rademacher, and D.~Ulbrich.
\newblock The eventual image of general {1D Greenberg-Hastings} cellular
  automata.
\newblock Preprint, 2018.

\bibitem{kolyada1996topological}
S.~Kolyada and L.~Snoha.
\newblock Topological entropy of nonautonomous dynamical systems.
\newblock {\em Random and Computational Dynamics}, 4(2):205, 1996.

\bibitem{krinsky1991wave}
V.~Krinsky and H.~Swinney.
\newblock Wave and patterns in biological and chemical excitable media.
\newblock {\em Noth-Holland, Amsterdam}, 1991.

\bibitem{lind1995introduction}
D.~Lind and B.~Marcus.
\newblock {\em An introduction to symbolic dynamics and coding}.
\newblock Cambridge university press, 1995.

\bibitem{:/content/aip/journal/chaos/16/3/10.1063/1.2266993}
N.~Manz and O.~Steinbock.
\newblock Propagation failures, breathing pulses, and backfiring in an
  excitable reaction-diffusion system.
\newblock {\em Chaos}, 16(3), 2006.

\bibitem{meyerovitch_2008}
T.~Meyerovitch.
\newblock Finite entropy for multidimensional cellular automata.
\newblock {\em Ergodic Theory and Dynamical Systems}, 28(4):1243--1260, 2008.

\bibitem{muller_spandl_2009}
J.~M\"uller and C.~Spandl.
\newblock Embeddings of dynamical systems into cellular automata.
\newblock {\em Ergodic Theory and Dynamical Systems}, 29(1):165--177, 2009.

\bibitem{muller_spandl_2010}
J.~M\"uller and C.~Spandl.
\newblock Embeddings of dynamical systems into cellular automata – erratum.
\newblock {\em Ergodic Theory and Dynamical Systems}, 30(4):1271–1271, 2010.

\bibitem{pearson1993complex}
J.~E. Pearson.
\newblock Complex patterns in a simple system.
\newblock {\em Science}, 261(5118):189--192, 1993.

\bibitem{PhysRevLett.71.1272}
J.-J. Perraud, A.~De~Wit, E.~Dulos, P.~De~Kepper, G.~Dewel, and P.~Borckmans.
\newblock One-dimensional ``spirals'': Novel asynchronous chemical wave
  sources.
\newblock {\em Phys. Rev. Lett.}, 71:1272--1275, Aug 1993.

\bibitem{petrov1994excitability}
V.~Petrov, S.~K. Scott, and K.~Showalter.
\newblock Excitability, wave reflection, and wave splitting in a cubic
  autocatalysis reaction-diffusion system.
\newblock {\em Philosophical Transactions of the Royal Society of London A:
  Mathematical, Physical and Engineering Sciences}, 347(1685):631--642, 1994.

\bibitem{reynolds1994dynamics}
W.~N. Reynolds, J.~E. Pearson, and S.~Ponce-Dawson.
\newblock Dynamics of self-replicating patterns in reaction diffusion systems.
\newblock {\em Physical review letters}, 72(17):2797, 1994.

\bibitem{shimomura1986}
T.~Shimomura.
\newblock Topological entropy and periodic points of a factor of a subshift of
  finite type.
\newblock {\em Nagoya Math. J.}, 104:117--127, 1986.

\bibitem{doi:10.1080/14689367.2017.1298724}
J.~Tang, B.~Li, and W.-C. Cheng.
\newblock Some properties on topological entropy of free semigroup action.
\newblock {\em Dynamical Systems}, 33(1):54--71, 2018.

\bibitem{DU16}
D.~Ulbrich.
\newblock Dynamics of the {1D Greenberg-Hastings} cellular automaton.
\newblock Master's thesis, Universit\"at Bremen, 2016.

\bibitem{walters2000introduction}
P.~Walters.
\newblock {\em An introduction to ergodic theory}, volume~79.
\newblock Springer Science \& Business Media, 2000.

\bibitem{WANG2009803}
W.~Yangeng, G.~Wei, and W.~H. Campbell.
\newblock Sensitive dependence on initial conditions between dynamical systems
  and their induced hyperspace dynamical systems.
\newblock {\em Topology and its Applications}, 156(4):803 -- 811, 2009.

\end{thebibliography}
\end{document}